\documentclass[british]{amsart}
\usepackage[a4paper,margin=28mm]{geometry}
\usepackage[T1]{fontenc}
\usepackage{lmodern}
\usepackage{amssymb,mathtools,tikz-cd}
\usepackage{microtype}
\usepackage[hidelinks]{hyperref}
\DeclareMathOperator{\GL}{GL}
\DeclareMathOperator{\Lie}{Lie}
\DeclareMathOperator{\ad}{ad}
\newtheorem{prop}{Proposition}[section]
\newtheorem{theorem}[prop]{Theorem}
\newtheorem{lemma}[prop]{Lemma}
\newtheorem{cor}[prop]{Corollary}
\newtheorem{mthm}{Theorem}

\theoremstyle{definition}
\newtheorem{definition}[prop]{Definition}
\newtheorem{example}[prop]{Example}
\theoremstyle{remark}
\newtheorem{remark}[prop]{Remark}
\newenvironment{proofof}[1]{\begin{proof}[Proof of #1]}{\end{proof}}
\title{Semigroup compactifications of algebraic groups}
\author{Elyasheev Leibtag}

\address{Department of Mathematics, Computer Science and Statistics, Ghent University, Krijgslaan 281-S9, 9000 Gent, Belgium}
\email{elyasheev.leibtag@ugent.be}

\urladdr{https://orcid.org/0000-0002-3463-5974}
\date{}
\hypersetup{pdftitle={Semigroup compactifications of algebraic groups},
  pdfauthor={Elyasheev Leibtag}}
\begin{document}
\begin{abstract}
We show that for connected affine algebraic groups over local fields
of characteristic zero, the following are equivalent: every continuous
homomorphism into a Hausdorff topological group has closed image,
every unitary representation decomposes into a direct sum of
finite-dimensional representations and representations that are mixing
modulo their kernels, and the matrix coefficients are uniformly dense
in the algebra of weakly almost periodic functions on the group.
In our proof, we employ methods from semigroup theory. We establish
that these groups are \emph{compactification-centric}, meaning
$sG=Gs$ for every element $s$ in the weakly almost periodic
compactification $w(G)$ of the group $G$.
\end{abstract}\maketitle

\section{Introduction}

In this paper, we study the connections between closed images of continuous homomorphisms, mixing properties of unitary representations, and weakly almost periodic functions. These connections have been studied for connected locally compact groups. We show that they also hold for connected affine algebraic groups over local fields of characteristic zero.

The first property we consider is a closed-image property. A topological group is called \emph{sealed} if every continuous homomorphism into a Hausdorff topological group has closed image. This property was studied for algebraic groups over local fields in \cite{bader2023homomorphic}, and characterised in terms of the algebraic structure of the group.

The second property concerns the decay of matrix coefficients. A strongly continuous unitary representation $(\pi,\mathcal H)$ of a locally compact group $G$ is \emph{mixing} if all its matrix coefficients belong to $C_0(G)$. The group $G$ has the \emph{Howe--Moore property} if every strongly continuous unitary representation without nonzero invariant vectors is mixing. 
The Howe--Moore theorem states that connected noncompact simple Lie groups with finite centre have this property.

The connection between the Howe--Moore property and closed-image properties was further developed by Bader and Gelander \cite{MR3692904}. Their main theorem on equicontinuous actions implies that semisimple algebraic groups over local fields are sealed and also gives a unified proof of the Howe--Moore theorem.

The third property comes from harmonic analysis. The matrix coefficients of strongly continuous unitary representations of $G$ form the Fourier--Stieltjes algebra $B(G)$. Its uniform closure
\[
        E(G)\coloneqq\overline{B(G)}^{\|\cdot\|_\infty}
\]
is called the \emph{Eberlein algebra}. Every matrix coefficient is
weakly almost periodic, so
\[
        E(G)\subseteq\operatorname{WAP}(G).
\]
The group $G$ is called \emph{Eberlein} if equality holds. This equality already fails for $G=\mathbb Z$, as shown by Rudin \cite{MR0102705}.

Examples of groups that are both sealed and Eberlein also arise in model theory. Ben Yaacov and Tsankov \cite{ben2016weakly} proved that automorphism groups of countable stable $\aleph_0$-categorical structures are totally minimal, and hence sealed. Under the additional assumption of $\aleph_0$-stability, these groups are also Eberlein, as shown by Ben Yaacov, Ibarluc\'ia, and Tsankov \cite{ben2018eberlein}. 

For connected locally compact groups, Mayer proved that the Eberlein property is equivalent to total minimality \cite{mayer1997asymptotics,MR1436849}. A topological group is \emph{totally minimal} if every Hausdorff quotient admits no strictly coarser Hausdorff group topology. For the groups $G=\mathbf G(k)$ considered in this paper, total minimality is equivalent to sealedness. Our main result gives the following equivalence for algebraic groups over local fields.

\begin{mthm}\label{mthm:main}
Let $k$ be a local field of characteristic zero and let $\mathbf G$ be a connected affine $k$-algebraic group. Put $G=\mathbf G(k)$. The following are equivalent:
\begin{enumerate}
    \item $G$ is sealed.
    \item Every strongly continuous unitary representation of $G$ is a Hilbert direct sum of representations $\pi_j$ such that, for each $j$, either $\pi_j$ is finite-dimensional or its induced representation of $G/\ker\pi_j$ is mixing, where the quotient carries the quotient topology.
    \item $G$ is Eberlein.
\end{enumerate}
\end{mthm}

By \cite[Theorem~8.1]{bader2023homomorphic}, sealedness has an algebraic characterisation. Given a $k$-Levi decomposition $\mathbf G=\mathbf L\ltimes R_u(\mathbf G)$, the group $G$ is sealed if and only if $Z(\mathbf L(k))$ is compact and the action of $\mathbf L(k)$ on $\Lie(R_u(\mathbf G))(k)$ has no nonzero invariant vectors. In our setting, sealedness is also equivalent to compactness of the centre of every quotient by a closed normal subgroup, see Theorem~\ref{thm:harmonic-equivalences}.

Our proofs use the theory of semitopological semigroup compactifications of groups. For a locally compact group $G$, the \emph{weakly almost periodic compactification} $w(G)$ is the Gelfand spectrum of $\operatorname{WAP}(G)$. It is a compact semitopological semigroup containing $G$ as an open dense subgroup. It is universal: every compact semitopological semigroup compactification $\varphi\colon G\to S$ extends uniquely to a continuous surjective semigroup homomorphism $w(G)\to S$. Here a compactification is a continuous homomorphism with dense image.
For a strongly continuous unitary representation $\pi$ of $G$, the weak-operator closure
\[
        S_\pi\coloneqq\overline{\pi(G)}^{\mathrm{WOT}}
\]
is a compact semitopological semigroup. By universality, $\pi$ extends to a continuous surjective semigroup homomorphism
\[
        w(G)\longrightarrow S_\pi.
\]

A theorem of Ruppert \cite[Chapter~III, Theorem~5.1]{ruppert2006compact} states that, for a connected locally compact group $G$, one has
\[
        sG=Gs
        \qquad(s\in w(G)).
\]
This motivates the following definition. A topological group $G$ is \emph{compactification-centric} if, for every compact semitopological semigroup compactification $\varphi\colon G\to S$,
\[
        s\varphi(G)=\varphi(G)s
        \qquad(s\in S).
\]
By universality, it suffices to verify this equality in $w(G)$.
\begin{mthm}\label{mthm:compactification-centric}
Let $k$ be a local field of characteristic zero and let $\mathbf G$ be a connected affine $k$-algebraic group. Then $G=\mathbf G(k)$ is compactification-centric.
\end{mthm}
Thus Ruppert's conclusion holds under algebraic connectedness, including over non-Archimedean local fields, where the groups of
$k$-points are totally disconnected.

We also formulate the Mautner phenomenon in terms of $w(G)$. We say that $G$ has the \emph{WAP Mautner property} if every $s\in w(G)\setminus G$ has a nontrivial right stabiliser
\[
        F_R(s)=\{g\in G:sg=s\}.
\]
We prove that the algebraic groups considered here have this property whenever their centre is compact (Theorem~\ref{thm:strong-wap-mautner}). For a compactification-centric group, the WAP Mautner property implies that every irreducible faithful unitary representation is mixing (Theorem~\ref{thm:irreducible-mixing}). In the topologically simple case, this method provides a unified proof of the Howe--Moore property.

To prove Theorem~\ref{mthm:compactification-centric}, we first prove compactification-centricity for solvable-by-compact algebraic groups by induction on dimension, using unipotent stabilisers obtained from the Mautner phenomenon. We then pass to the general case using the cocompactness of minimal $k$-parabolic subgroups and generation by their conjugates.

\subsection*{Acknowledgements}
The author is grateful to Uri Bader for guidance and many helpful
suggestions. The author also thanks Omer Lavi, Guy Salomon, Itamar
Vigdorovich, and the participants of seminars where preliminary versions
of this work were discussed.

\subsection*{Use of AI}
The main results of this paper appeared in the 2024 arXiv
version. During the preparation of the present revision,
the author used ChatGPT-5.6 Sol for assistance with language and
\LaTeX{} editing, including assistance with the formulation and
proof of Lemma~\ref{lem:central-scalar-limit}. The author
has reviewed the revisions and takes full responsibility
for the final manuscript.

\section{Preliminaries}
\label{section:prelim}

All topological groups and spaces considered below are Hausdorff. By a compact semitopological semigroup compactification of a group $G$ we mean a continuous homomorphism $\varphi\colon G\to S$ with dense image in a compact semitopological semigroup. The homomorphism is allowed to have a nontrivial kernel. We recall the semigroup-theoretic facts used below, following Ruppert's monograph \cite{ruppert2006compact}.

We write $1_G$ for the identity of a group $G$, reserving $e$ for a semigroup idempotent when both occur in an argument. Nets are indexed by directed sets: for $(x_i)_{i\in\mathcal I}$, we write $x_i\to x$ or $\lim_i x_i=x$. After passing to a subnet, we retain the same notation unless stated otherwise. For a net in a locally compact space, $x_i\to\infty$ means that it eventually leaves every compact subset. Closures in a semigroup compactification are taken in that compactification unless another ambient space is specified.

\subsection{Compact semitopological semigroups}

A \emph{compact semitopological semigroup} is a compact Hausdorff space $S$ equipped with an associative multiplication such that the maps
\[
        x\longmapsto sx,
        \qquad
        x\longmapsto xs
\]
are continuous for every $s\in S$.  We write $J(S)$ for the set of idempotents of $S$.  This set is nonempty.  For $e\in J(S)$, there is a unique maximal subgroup $H_e$ of $S$ having $e$ as its identity.  If the idempotents commute, we use the standard partial order
\[
        e\leq f
        \quad\Longleftrightarrow\quad
        ef=e.
\]

Let $S$ contain a dense subgroup $G$.  The Ryll-Nardzewski theorem asserts that $S$ has a unique minimal ideal $M(S)$ and that $M(S)$ is a compact topological group. 

The Ellis--Lawson continuity theorem asserts that every separately continuous action of $S$ on a compact Hausdorff space $X$ is jointly continuous at every point $(s,x_0)$ such that $Sx_0=X$ \cite[II.4.11]{ruppert2006compact}. In particular, multiplication in $S$ is jointly continuous at
every point of
\[
        (G\times S)\cup(S\times G).
\]

Let $N\lhd G$ be a closed normal subgroup, and let $\overline N$ denote its closure in $S$. Let $M=M(\overline N)$ be the minimal ideal of $\overline N$.  The quotient relation
\[
        s_1\sim_M s_2
        \quad\Longleftrightarrow\quad
        s_1M=s_2M
\]
turns $S/M$ into a compact semitopological semigroup, see \cite[Chapter~III, Proposition~1.10]{ruppert2006compact}.

\begin{remark}\label{rem:cocompact-decomposition}
Let $G$ be a locally compact group embedded densely in a compact semitopological semigroup $S$, and let $H\leq G$ be a closed cocompact subgroup. Then
\[
        S=G\overline H=\overline H G.
\]
Indeed, compactness of $G/H$ and openness of the quotient map provide a compact subset $K\subseteq G$ with $G=KH$. Joint continuity on $G\times S$ shows that $K\overline H$ is compact. Since it contains the dense subset $G$, it equals $S$. Taking inverses in $G$ gives $G=HK^{-1}$, and the same argument gives $S=\overline H K^{-1}$.
\end{remark}

\subsection{The WAP compactification}

Let $C_b(G)$ be the Banach space of bounded continuous complex-valued functions on $G$.  A function $f\in C_b(G)$ is \emph{weakly almost periodic} if its orbit under left translation is relatively weakly compact in $C_b(G)$.  The weakly almost periodic functions form a unital commutative $C^*$-algebra, denoted $\operatorname{WAP}(G)$.
\begin{definition}[The WAP Compactification]
Let $G$ be a topological group. The Gelfand spectrum, or equivalently the
maximal ideal space, of $\operatorname{WAP}(G)$, denoted by $w(G)$, is called the
\emph{weakly almost periodic (WAP) compactification} of $G$.
\end{definition}

Let
\[
    \omega\colon G\longrightarrow w(G)
\]
denote the canonical map given by evaluation. In this
compactification, the multiplication on $G$ extends uniquely to a
separately continuous multiplication on $w(G)$. Thus $w(G)$ is a compact
semitopological semigroup. If $G$ is locally compact, then $\omega$ is a
topological embedding with open dense image
\cite[Chapter~III, Theorem~3.3]{ruppert2006compact}. We henceforth identify
$G$ with its image in $w(G)$.

The WAP compactification is universal among compact semitopological semigroup compactifications. More precisely, if
$
\varphi\colon G\longrightarrow S
$
is a continuous homomorphism into a compact semitopological semigroup, then there is a unique continuous semigroup homomorphism
$
\widetilde{\varphi}\colon w(G)\longrightarrow S
$
making the following diagram commute \cite[Chapter~III, Theorem~3.3]{ruppert2006compact}:
\[
\begin{tikzcd}
G \arrow{r}{\varphi} \arrow[swap]{d}{\omega} & S \\
w(G) \arrow[swap]{ur}{\widetilde{\varphi}} &
\end{tikzcd}
.
\]
The WAP compactification is well-behaved with respect to quotient groups. Let $N\lhd G$ be a closed normal subgroup, and let $\overline{N}$ denote its closure in $w(G)$. If
\[
    M=M(\overline{N})
\]
is the minimal ideal of $\overline{N}$, then
\[
    w(G)/M\cong w(G/N)
\]
as compact semitopological semigroups
\cite[Chapter~III, Proposition~1.10]{ruppert2006compact}.
\begin{lemma}\label{lem:lift-group-point}
Let $G$ be a locally compact group, let $N\lhd G$ be closed, and let
\[
        q_N\colon w(G)\longrightarrow w(G/N)
\]
be the canonical morphism. For every $s\in w(G)$ and $g\in G$,
\[
        q_N(s)=gN
        \quad\Longleftrightarrow\quad
        s\in g\overline N.
\]
Equivalently, $q_N^{-1}(\{gN\})=g\overline N$.
\end{lemma}

\begin{proof}
Continuity of $q_N$ gives $q_N(\overline N)=\{N\}$, so $s\in g\overline N$ implies $q_N(s)=gN$.

Conversely, suppose that $q_N(s)=gN$. Replacing $s$ by $g^{-1}s$, we may assume that $q_N(s)=N$. Choose a net $(x_i)_{i\in\mathcal I}$ in $G$ converging to $s$, and let $p_N\colon G\to G/N$ be the quotient map. The diagram
\[
\begin{tikzcd}
G \arrow{r}{\omega} \arrow[swap]{d}{p_N} & w(G) \arrow{d}{q_N} \\
G/N \arrow{r}{\omega_{G/N}} & w(G/N)
\end{tikzcd}
\]
commutes. Since $G/N$ is locally compact and is embedded in $w(G/N)$, we have $p_N(x_i)\to 1_{G/N}$ in $G/N$. Refining the net, we may write
\[
        x_i=u_i n_i,
        \qquad
        u_i\longrightarrow 1_G,
        \qquad
        n_i\in N.
\]
Joint continuity at $(1_G,s)$ now gives $n_i=u_i^{-1}x_i\to s$. Hence $s\in\overline N$.
\end{proof}

\subsection{Weak-operator compactifications}

We next describe an important family of compactifications arising from unitary representations. Let
\[
    \pi\colon G\longrightarrow\mathcal U(\mathcal H)
\]
be a strongly continuous unitary representation. Its weak-operator closure
\[
    S_\pi
    \coloneqq
    \overline{\pi(G)}^{\mathrm{WOT}}
    \subseteq \mathcal B(\mathcal H)
\]
is a compact semitopological semigroup consisting of contractions. We call $S_\pi$ the \emph{weak-operator compactification associated with $\pi$}, see \cite[VI.2.12]{berglund1989analysis}. The relationship between these compactifications and the uniformly closed algebras generated by their matrix coefficients is studied in detail by Spronk and Stokke \cite{spronk2013matrix}.

The map $\pi\colon G\to S_\pi$ is continuous for the weak operator topology. Hence the universal property of $w(G)$ makes $S_\pi$ a factor of the WAP compactification.

For every $\xi,\eta\in\mathcal H$, the function
\[
T\longmapsto \langle T\xi,\eta\rangle,
\qquad T\in S_\pi,
\]
is continuous. Composing with $\pi\colon G\to S_\pi$ gives the matrix coefficient
\[
g\longmapsto \langle \pi(g)\xi,\eta\rangle.
\]
\begin{example}[The left regular representation]
Let $G$ be a noncompact locally compact group, and let
\[
    \lambda_G\colon G\longrightarrow\mathcal U(L^2(G))
\]
be its left regular representation. The associated weak-operator compactification is the one-point compactification of $G$. More precisely,
\[
    S_{\lambda_G}
    =
    \overline{\lambda_G(G)}^{\mathrm{WOT}}
    =
    \lambda_G(G)\cup\{0\},
\]
where $0$ denotes the zero operator. Indeed, the matrix coefficients of $\lambda_G$ belong to $C_0(G)$. Consequently, $ \lambda_G(g)\longrightarrow 0$ in the weak operator topology as $g\to\infty$.
\end{example}
\begin{remark}\label{rem:matrix_coef_WAP}
Let $ \pi\colon G\to\mathcal U(\mathcal H)$ be a strongly continuous unitary representation. Every matrix coefficient
\[
g\longmapsto \langle\pi(g)\xi,\eta\rangle
\]
belongs to $\operatorname{WAP}(G)$. Indeed, its translates are obtained by replacing one of the vectors by vectors in the bounded orbit $\pi(G)\xi$ or $\pi(G)\eta$, which is relatively weakly compact in $\mathcal H$. Hence, by the definition of $w(G)$, every matrix coefficient extends uniquely to a continuous function on $w(G)$.
Explicitly, if
\[
\widetilde\pi\colon w(G)\longrightarrow S_\pi
\]
is the canonical projection to $S_\pi$, then this extension is given by
\[
s\longmapsto
\left\langle \widetilde\pi(s)\xi,\eta\right\rangle.
\]
\end{remark}

\section{Compactification-centric groups}
\label{section:compactification_centric}

We now consider locally compact groups for which the left and right $G$-orbits of every element of the WAP compactification coincide.

Ruppert proved that every connected locally compact group satisfies
\[
        sG=Gs
        \qquad
        \text{for every }s\in w(G),
\]
see \cite[Chapter~III, Theorem~5.1]{ruppert2006compact}. This result motivates the following definition.
\begin{definition}\label{def:compactification_centric}
A topological group $G$ is called \emph{compactification-centric} if, for every compact semitopological semigroup compactification $\varphi\colon G\to S$, one has
\[
        s\varphi(G)=\varphi(G)s
        \qquad \text{for every } s\in S.
\]
\end{definition}

By the universal property of $w(G)$, it is enough to verify the equality in the WAP compactification.

\begin{prop}\label{prop:cor_of_compactification_centric}
Let $G$ be compactification-centric and let $\varphi\colon G\to S$ be a semitopological semigroup compactification.  Then the following hold.
\begin{enumerate}
    \item Every one-sided ideal of $S$ is two-sided.
    \item For every $s\in S$, the subgroup
    \[
            F_R(s)=\{g\in G:s\varphi(g)=s\}
    \]
    is closed and normal in $G$.
    \item Every idempotent of $S$ is central.
    \item For every idempotent $e\in J(S)$, the subset $\varphi(G)\cdot e$ is a group with identity $e$.
\end{enumerate}
\end{prop}

\begin{proof}
Write $H=\varphi(G)$, which is dense in $S$. By compactification-centricity, $sH=Hs$ for every $s\in S$. Separate continuity and compactness therefore give
\[
        sS=\overline{sH}=\overline{Hs}=Ss.
\]
\begin{enumerate}
    \item Let $I\subseteq S$ be a right ideal. For every $a\in I$,
    \[
            Sa=aS\subseteq I.
    \]
    Hence $SI\subseteq I$, so $I$ is also a left ideal. The argument for a left ideal is identical.

    \item The right stabiliser $F_R(s)$ is a closed subgroup of $G$. Let $n\in F_R(s)$ and $g\in G$. Since $sH=Hs$, there is $h\in G$ such that $s\varphi(g)=\varphi(h)s$. Therefore
    \[
    \begin{aligned}
            s\varphi(gng^{-1})
            &=\varphi(h)s\varphi(n)\varphi(g^{-1})\\
            &=\varphi(h)s\varphi(g^{-1})
            =s.
    \end{aligned}
    \]
    Thus $F_R(s)$ is normal in $G$.

    \item Let $e\in J(S)$ and $t\in S$. Since $eS=Se$, there are $t_1,t_2\in S$ such that
    \[
            et=t_1e,
            \qquad
            te=et_2.
    \]
    Using $e^2=e$, we obtain
    \[
            te=e(et_2)=ete=(t_1e)e=et.
    \]
    Hence $e$ is central.

    \item Since $e$ is central, for every $g,h\in G$,
    \[
            (\varphi(g)e)(\varphi(h)e)=\varphi(gh)e.
    \]
    Thus $He$ is a group with identity $e=\varphi(1_G)e$ and inverse
    \[
            (\varphi(g)e)^{-1}=\varphi(g^{-1})e.
    \]
\end{enumerate}
\end{proof}

\begin{lemma}\label{lem:quotient-cc}
Let $G$ be compactification-centric, and let $N$ be a closed normal subgroup of $G$. Then $G/N$ is compactification-centric.
\end{lemma}

\begin{proof}
Let $M=M(\overline N)$ be the minimal ideal of the closure of $N$ in $w(G)$.  The quotient $w(G)/M$ is canonically isomorphic to $w(G/N)$.
Since $sG=Gs$ in $w(G)$, their images in $w(G)/M$ coincide.  Hence $G/N$ is compactification-centric.
\end{proof}

\begin{lemma}\label{lem:fixed-normal-quotient}
Let $G$ be locally compact and $N\lhd G$ a closed subgroup. Let $s\in w(G)$. If $G/N$ is compactification-centric and $sN=s$, then
\[
        sG=Gs.
\]
The analogous assertion holds if $Ns=s$.
\end{lemma}

\begin{proof}
Let $M=M(\overline N)$.  Since $N\lhd G$, the group $G$ normalises $\overline N$ and hence also its unique minimal ideal $M$.  Thus
\[
        gM=Mg
        \qquad(g\in G).
\]
Separate continuity gives $s\overline N=s$ and hence $sM=s$.
Compactification-centricity of $G/N$ gives
\[
        sGM=GsM.
\]
Since
\[
        sGM=sMG=(sM)G=sG
\]
and
\[
        GsM=G(sM)=Gs,
\]
we obtain $sG=Gs$.  The other assertion is symmetric.
\end{proof}

\begin{remark}
For a compactification $\varphi\colon G\to S$ and $s\in S$, we similarly write
\[
        F_L(s)\coloneqq\{g\in G:\varphi(g)s=s\}.
\]
In $w(G)$, this becomes $F_L(s)=\{g\in G:gs=s\}$.
\end{remark}

\begin{definition}\label{def:wap-mautner}
A locally compact group $G$ has the \emph{WAP Mautner property} if
\[
        F_R(s)\neq\{1_G\}
        \qquad
        \text{for every }s\in w(G)\setminus G.
\]
\end{definition}

Recall that a unitary representation is mixing if all of its matrix
coefficients belong to $C_0(G)$. A locally compact group has the
\emph{Howe--Moore property} if every strongly continuous unitary
representation without nonzero invariant vectors is mixing.
The following theorem is the basic representation-theoretic consequence.

\begin{theorem}\label{thm:irreducible-mixing}
Let $G$ be compactification-centric and have the WAP Mautner property. Then every strongly continuous irreducible faithful unitary representation of $G$ is mixing.

Moreover, if $G$ is noncompact and topologically simple, then $G$ has the Howe--Moore property.
\end{theorem}
\begin{proof}
Let $\pi\colon G\to\mathcal U(\mathcal H)$ be an irreducible faithful unitary representation, and let
\[
\widetilde\pi\colon w(G)\longrightarrow S_\pi
\]
be the canonical projection.
Fix $s\in w(G)\setminus G$ and put $T=\widetilde\pi(s)$. By the WAP Mautner property, there exists $g\neq 1_G$ such that $T\pi(g)=T$. Since $\pi$ is faithful, $\pi(g)-I\neq0$ and
\[
        \{0\}\neq\operatorname{ran}(\pi(g)-I)\subseteq\ker T.
\]
By compactification-centricity $T\pi(G)=\pi(G)T$.
It follows that $\ker T$ is $G$-invariant.
By irreducibility, $\ker T=\mathcal H$, hence $T=0$. We have therefore shown
that
\[
\widetilde\pi(s)=0
\qquad
\text{for every }s\in w(G)\setminus G.
\]
By Remark~\ref{rem:matrix_coef_WAP}, every matrix coefficient of $\pi$ extends continuously to $w(G)$. Hence every matrix coefficient belongs to $C_0(G)$, and $\pi$ is mixing.

For the moreover statement, assume that $G$ is noncompact and topologically simple. Let $\pi\colon G\to\mathcal U(\mathcal H)$ be a unitary representation with $\mathcal H^G=\{0\}$, and denote its extension to $w(G)$ by $\widetilde\pi$. Fix $s\in w(G)\setminus G$. By Proposition~\ref{prop:cor_of_compactification_centric}, $F_R(s)$ is a closed normal subgroup. It is nontrivial by the WAP Mautner property, so $F_R(s)=G$. Taking $T=\widetilde\pi(s)$, we obtain
\[
T\pi(g)=T
\qquad
\text{for every }g\in G.
\]
Taking adjoints gives
\[
\pi(g^{-1})T^*=T^*
\qquad
\text{for every }g\in G.
\]
Thus $\operatorname{ran}(T^*)\subseteq\mathcal H^G=\{0\}$, so $T=0$. Therefore $\widetilde\pi$ vanishes on $w(G)\setminus G$ as needed.
\end{proof}

Together with the compactification-centricity and WAP
Mautner results proved below, Theorem~\ref{thm:irreducible-mixing}
recovers the classical Riemann--Lebesgue lemma.

\begin{cor}[Riemann--Lebesgue lemma]
Let $K$ be a local field of characteristic zero, and let
$\psi$ be a nontrivial continuous additive character of $K$.
For every $f\in L^1(K)$,
\[
        \widehat f(b)\coloneqq\int_K f(x)\psi(bx)\,dx
        \longrightarrow0
        \qquad\text{as }|b|\longrightarrow\infty.
\]
\end{cor}

\begin{proof}
Put $G=K\rtimes K^\times$. The representation
$\pi=\operatorname{Ind}_K^G\psi$ admits the realisation
on $L^2(K)$ given by
\[
        (\pi(b,a)\xi)(x)
        =|a|^{1/2}\psi(bx)\xi(ax),
\]
where $|\cdot|$ is the normalised absolute value.

The representation $\pi$ is irreducible by Mackey's
criterion \cite[Theorem~6.43]{folland2016course}, since the
stabiliser of $\psi$ in $K^\times$ is trivial. The representation is
also faithful: $\pi(b,a)=I$ forces $a=1$ by its action on
supports, and then $\psi(bx)=1$ for every $x\in K$, so $b=0$.

By Theorems~\ref{thm:compactification-centric}
and~\ref{thm:strong-wap-mautner}, the group $G$ is
compactification-centric and has the WAP Mautner property,
since its centre is trivial. Thus
Theorem~\ref{thm:irreducible-mixing} implies that $\pi$
is mixing.
Write $f=\xi\cdot \overline\eta$ with
$\xi,\eta\in L^2(K)$. Then
\[
        \widehat f(b)
        =\langle\pi(b,1)\xi,\eta\rangle
        \longrightarrow0,
\]
since $(b,1)$ leaves every compact subset of $G$ as
$|b|\to\infty$.
\end{proof}
\begin{remark}[Closed images and the Howe--Moore property]
\label{rem:zimmer-closed-images}
Zimmer observed that the Howe--Moore property yields a topological
strengthening of simplicity \cite[Theorem~7.3]{zimmer1977orbit}:
for a noncompact topologically simple locally compact group $G$
with this property, every nontrivial continuous homomorphism into
a locally compact group is a topological isomorphism onto a closed
subgroup.

For a nontrivial dense homomorphism $f\colon G\to H$, the group $H$
is noncompact, since $G$ has no nontrivial finite-dimensional unitary
representations, and $\pi=\lambda_H\circ f$ is faithful and mixing.
Thus $S_\pi=\pi(G)\cup\{0\}$ identifies with the one-point
compactification of $G$, making $\pi$ a topological embedding
with closed image in $\mathcal U(L^2(H))$, hence $f$ is a topological
isomorphism onto $H$.
\end{remark}
\section{Algebraic groups are compactification-centric}
\label{section:alg_are_compactification_centric}

In this section, we prove Theorem~\ref{thm:compactification-centric} and establish the strong form of the WAP Mautner property for $k$-algebraic groups.

We restate Theorem~\ref{mthm:compactification-centric}.

\begin{theorem}
\label{thm:compactification-centric}
Let $k$ be a local field of characteristic zero and let $\mathbf G$ be a connected affine $k$-algebraic group.  Then $G=\mathbf G(k)$ is compactification-centric.
\end{theorem}

Throughout this section, $k$ is a local field of characteristic zero,
$\mathbf G$ is a connected affine $k$-algebraic group, and
$G=\mathbf G(k)$. Its usual topology is locally compact and second
countable, and hence Polish. We use the notation
\[
        \mathbf U\coloneqq R_u(\mathbf G),
        \qquad U\coloneqq\mathbf U(k),
        \qquad \mathfrak u\coloneqq\Lie(\mathbf U)(k),
        \qquad V\coloneqq Z(\mathfrak u).
\]
The algebraic exponential and logarithm maps induce mutually inverse maps
\[
        \exp\colon\mathfrak u\longrightarrow U,
        \qquad
        \log\colon U\longrightarrow\mathfrak u,
\]
which are equivariant for conjugation, see \cite[I.1.3.2]{margulis1991discrete}. We denote the resulting action on $\mathfrak u$ by
\[
        \rho_{\mathfrak u}\colon G\longrightarrow\GL(\mathfrak u),
        \qquad
        \rho_{\mathfrak u}(g)=\operatorname{Ad}_G(g)|_{\mathfrak u}.
\]
Thus
\[
        g\exp(X)g^{-1}=\exp\bigl(\rho_{\mathfrak u}(g)X\bigr)
        \qquad(g\in G,\ X\in\mathfrak u).
\]
We use $\operatorname{Ad}_G$ for the adjoint action on $\Lie(\mathbf G)(k)$ and $\ad_v(u)=[v,u]$ for the Lie algebra adjoint map. Set
\[
        C=\ker\rho_{\mathfrak u}=C_G(U).
\]
The equality follows from equivariance and bijectivity of $\exp$. In particular, $C$ is the group of $k$-points of the algebraic centraliser $C_{\mathbf G}(\mathbf U)$.

For $s\in w(G)$, define the right and left Lie-fixed sets by
\[
\begin{aligned}
        \mathfrak F_R(s)
        &\coloneqq\{X\in\mathfrak u:\exp(\lambda X)\in F_R(s)
          \text{ for every }\lambda\in k\},\\
        \mathfrak F_L(s)
        &\coloneqq\{X\in\mathfrak u:\exp(\lambda X)\in F_L(s)
          \text{ for every }\lambda\in k\}.
\end{aligned}
\]
Their intersections with $V$ are $k$-vector subspaces: scalar multiplication follows from the definition, and addition follows from
\[
        \exp(\lambda(X+Y))=\exp(\lambda X)\exp(\lambda Y)
        \qquad(X,Y\in V,\ \lambda\in k).
\]
When normalising vectors, we use maximum norms relative to $k$-bases. With this choice, every nonzero vector can be rescaled by an element of $k^\times$ to have norm one, and the unit sphere is compact.

\begin{lemma}\label{lem:commutator-propagation}
Let $s\in w(G)$ and let $v\in\mathfrak F_R(s)$.  If $u\in\mathfrak u$
and, for some $m\geq2$,
\[
        \ad_v^m(u)=0,
\]
then
\[
        \ad_v^{m-1}(u)\in\mathfrak F_R(s).
\]
\end{lemma}

\begin{proof}
Choose a net $(\lambda_i)_{i\in\mathcal I}$ in $k^\times$ with $|\lambda_i|\to\infty$.  After
passing to a subnet, assume that
\[
        A_i\coloneqq\exp(-\lambda_i v)\longrightarrow r\in w(G).
\]
Fix $\lambda\in k$ and put
\[
        B_i=
        \exp\bigl(\lambda_i^{-(m-1)}\lambda u\bigr).
\]
Since $\ad_v^m(u)=0$,
\[
 \rho_{\mathfrak u}(\exp(\lambda_i v))
   \bigl(\lambda_i^{-(m-1)}\lambda u\bigr)
 =\sum_{j=0}^{m-2}
   \lambda_i^{j-(m-1)}\frac{\lambda}{j!}\ad_v^j(u)+\frac{\lambda}{(m-1)!}\ad_v^{m-1}(u),
\]
and hence
\[
        A_i^{-1}B_iA_i
        \longrightarrow
        \exp\!\left(
          \frac{\lambda}{(m-1)!}\ad_v^{m-1}(u)
        \right).
\]
Since $m-1\geq1$, we have $B_i\to1_G$. Applying Ellis--Lawson joint continuity to $A_i(A_i^{-1}B_iA_i)=B_iA_i$ yields
\[
        r\exp\!\left(
          \frac{\lambda}{(m-1)!}\ad_v^{m-1}(u)
        \right)=r.
\]
Since $v\in\mathfrak F_R(s)$, one has $sA_i=s$ for every $i$ and $sr=s$. Multiplying the preceding equality by $s$ on the left gives
\[
        s\exp\!\left(
          \frac{\lambda}{(m-1)!}\ad_v^{m-1}(u)
        \right)=s.
\]
As $\lambda$ is arbitrary, the conclusion follows.
\end{proof}

\begin{cor}\label{cor:fixed-central-vector}
If $\mathfrak F_R(s)\neq\{0\}$, then
\[
        \mathfrak F_R(s)\cap Z(\mathfrak u)\neq\{0\}.
\]
\end{cor}

\begin{proof}
Let
\[
        0=Z_0(\mathfrak u)\subset Z_1(\mathfrak u)\subset\cdots
        \subset Z_c(\mathfrak u)=\mathfrak u
\]
be the upper central series.  Choose $0\neq v\in\mathfrak F_R(s)$ with minimal $j$ such that $v\in Z_j(\mathfrak u)$.  We will show that $j=1$.
Suppose, for a contradiction, that $j>1$.  Since $v\notin Z_{j-1}(\mathfrak u)$, there is $u\in\mathfrak u$ with $[v,u]\neq0$.  Let $m$ be maximal such that $\ad_v^{m-1}(u)\neq0$.
Notice that $[v,u] \neq 0$ forces $m \geq 2$.
Lemma~\ref{lem:commutator-propagation} shows that
\[
        0\neq\ad_v^{m-1}(u)\in\mathfrak F_R(s).
\]
Moreover, commutation with $v\in Z_j(\mathfrak u)$ lowers the upper-central degree, so this vector belongs to $Z_{j-1}(\mathfrak u)$.  This contradicts the minimality of $j$.  Hence $j=1$.
\end{proof}

\begin{lemma}\label{lem:linear-divergence}
Let $V$ be a finite-dimensional normed vector space over a local field, and let $(A_i)_{i\in\mathcal I}$ be a net in $\GL(V)$ escaping every compact subset. After passing to a subnet, one of the following holds:
\begin{enumerate}
    \item There is $0\neq v\in V$ such that $\|A_i v\|\to0$.
    \item There is $0\neq v\in V$ such that $\|A_i v\|\to\infty$.
\end{enumerate}
\end{lemma}

\begin{proof}
After passing to a subnet, either $\|A_i\|\to\infty$ or $(A_i)$ is bounded in $\operatorname{End}(V)$.

In the first case, fix a basis of $V$. Since the operator norm is bounded above by a constant times the largest norm of the images of the basis vectors, a further subnet admits a fixed basis vector $v$ with $\|A_i v\|\to\infty$.

In the second case, compactness of closed bounded subsets of $\operatorname{End}(V)$ gives a further subnet with $A_i\to A$. The limit $A$ cannot be invertible: otherwise the net would eventually lie in a compact neighbourhood of $A$ in $\GL(V)$, contrary to the hypothesis. Choose $0\neq v\in\ker A$. Then $A_i v\to0$.
\end{proof}

The next lemma is the Mautner step used throughout the proof.

\begin{lemma}\label{lem:adjoint-divergence-fixed}
Let $(g_i)_{i\in\mathcal I}$ be a net in $G$ converging to $s\in w(G)$. If $(\rho_{\mathfrak u}(g_i))$ escapes every compact subset of $\GL(\mathfrak u)$, then
\[
        \mathfrak F_R(s)\neq\{0\}.
\]
\end{lemma}

\begin{proof}
Put $A_i=\rho_{\mathfrak u}(g_i)$. By Lemma~\ref{lem:linear-divergence}, after passing to a subnet either $A_i v\to0$ or $\|A_i v\|\to\infty$ for some fixed $0\neq v\in\mathfrak u$.

In the first case, for every $\lambda\in k$,
\[
        g_i\exp(\lambda v)g_i^{-1}
        =\exp(\lambda A_i v)\longrightarrow 1_G.
\]
Ellis--Lawson joint continuity therefore gives
\[
        s\exp(\lambda v)
        =\lim_i\bigl[(g_i\exp(\lambda v)g_i^{-1})g_i\bigr]
        =s.
\]
Thus $v\in\mathfrak F_R(s)$.

Assume now that $\|A_i v\|\to\infty$. Normalising $A_i v$ and passing to a subnet, we find scalars $\gamma_i\in k^\times$ and $0\neq v'\in\mathfrak u$ such that
\[
        \gamma_i\longrightarrow0,
        \qquad A_i(\gamma_i v)\longrightarrow v'.
\]
For every $\lambda\in k$,
\[
\begin{aligned}
        \exp(\lambda v')s
        &=\lim_i\bigl[(g_i\exp(\lambda\gamma_i v)g_i^{-1})g_i\bigr]\\
        &=\lim_i\bigl[g_i\exp(\lambda\gamma_i v)\bigr]
        =s.
\end{aligned}
\]
Hence $v'\in\mathfrak F_L(s)$.

After passing to a further subnet, there are three possibilities for $A_i^{-1}v'$.

First, suppose that $\|A_i^{-1}v'\|\to\infty$. Normalisation and passage to a subnet give scalars $\delta_i\to0$ and $0\neq w\in\mathfrak u$ such that
\[
        A_i^{-1}(\delta_i v')\longrightarrow w.
\]
For every $\lambda\in k$,
\[
\begin{aligned}
        s\exp(\lambda w)
        &=\lim_i\bigl[g_i(g_i^{-1}\exp(\lambda\delta_i v')g_i)\bigr]\\
        &=\lim_i\bigl[\exp(\lambda\delta_i v')g_i\bigr]
        =s.
\end{aligned}
\]
Thus $w\in\mathfrak F_R(s)$.

Second, suppose that $A_i^{-1}v'\to w\neq0$. For every $\lambda\in k$,
\[
        \exp(\lambda v')g_i
        =g_i\exp(\lambda A_i^{-1}v').
\]
The left-hand side converges to $s$, since $v'\in\mathfrak F_L(s)$, and the right-hand side converges to $s\exp(\lambda w)$. Hence again $w\in\mathfrak F_R(s)$.

Finally, suppose that $A_i^{-1}v'\to0$. Normalising these nonzero vectors and passing to a subnet, choose scalars $\beta_i\in k^\times$ and $0\neq\widetilde v\in\mathfrak u$ such that
\[
        A_i^{-1}(\beta_i v')\longrightarrow\widetilde v.
\]
Let
\[
        H=\{\exp(\mu v'):\mu\in k\}.
\]
Since $v'\in\mathfrak F_L(s)$, separate continuity gives $\overline Hs=s$. Let $M=M(\overline H)$, and let $e$ be the identity of the compact group $M$. In particular, $es=s$.

Fix $\lambda\in k$. After passing to a further subnet, which may depend on $\lambda$, assume that
\[
        \exp(\lambda\beta_i v')\longrightarrow z_\lambda\in\overline H.
\]
The elements $e\exp(\lambda\beta_i v')e$ belong to $M$ and converge to $ez_\lambda e$. Also,
\[
        g_i^{-1}\exp(\lambda\beta_i v')g_i
        \longrightarrow\exp(\lambda\widetilde v).
\]
Since $\overline H$ is commutative and $e^2=e$,
\[
\begin{aligned}
        eg_i\bigl(g_i^{-1}\exp(\lambda\beta_i v')g_i\bigr)
        &=e\exp(\lambda\beta_i v')g_i\\
        &=\bigl(e\exp(\lambda\beta_i v')e\bigr)(eg_i).
\end{aligned}
\]
The compact topological group $M$ acts by left multiplication on the compact space $ew(G)$, with $e$ acting as the identity. This separately continuous action is jointly continuous by the Ellis joint continuity theorem. Combining this with joint continuity on $w(G)\times G$ yields
\[
        s\exp(\lambda\widetilde v)=(ez_\lambda e)s=s.
\]
Since $\lambda$ was arbitrary, $\widetilde v\in\mathfrak F_R(s)$.
\end{proof}

\begin{remark}\label{rem:left-fixed}
Applying a symmetric argument gives the corresponding nontriviality statement
for $\mathfrak F_L(s)$.
\end{remark}

We next make the fixed subspace invariant under a split torus.

\begin{lemma}\label{lem:torus-invariant-fixed-space}
Let $\mathbf T\leq\mathbf G$ be a $k$-split torus, put $T=\mathbf T(k)$, and let $s\in w(G)$. If $\mathfrak F_R(s)\neq\{0\}$, then $\mathfrak F_R(s)\cap V$ contains a nonzero $T$-invariant $k$-subspace. This subspace may be chosen inside a single $T$-weight space.
\end{lemma}

\begin{proof}
By Corollary~\ref{cor:fixed-central-vector}, the $k$-subspace
\[
        W=\mathfrak F_R(s)\cap V
\]
is nonzero. Let $X_k^*(\mathbf T)$ denote the group of algebraic characters of $\mathbf T$ defined over $k$, and write
\[
        V=\bigoplus_{\alpha\in\Omega}V_\alpha,
\]
where $\Omega\subseteq X_k^*(\mathbf T)$ is the finite set of weights occurring in $V$ and
\[
        V_\alpha\coloneqq
        \{v\in V:\rho_{\mathfrak u}(t)v=\alpha(t)v
          \text{ for every }t\in T\}.
\]
The set $\Omega$ includes the trivial character if it occurs.
Suppose that $W$ contains no nonzero vector in a single weight space. Choose
\[
        0\neq u=\sum_{j=1}^r u_j\in W,
        \qquad 0\neq u_j\in V_{\alpha_j},
\]
with $r$ minimal, where $\alpha_1,\ldots,\alpha_r$ are pairwise distinct members of $\Omega$. Then $r\geq2$. Since distinct characters of a split torus are separated by cocharacters, choose a $k$-cocharacter $\chi\colon\mathbf G_m\to\mathbf T$ such that the integers
\[
        m_j\coloneqq\langle\alpha_j,\chi\rangle,
        \qquad \alpha_j(\chi(t))=t^{m_j}\quad(t\in k^\times),
\]
are not all equal. Choose a net $(\lambda_i)_{i\in\mathcal I}$ in $k^\times$ with $|\lambda_i|\to\infty$, and put
\[
        t_i=\chi(1+\lambda_i^{-1}).
\]
Then $t_i\to1_G$ and
\[
        \rho_{\mathfrak u}(t_i)(\lambda_i u)-\lambda_i u
        \longrightarrow
        u'\coloneqq\sum_{j=1}^r m_j u_j.
\]
After passing to a subnet, we may assume that $\exp(\lambda_i u)\to r_0\in w(G)$. Since $u\in W$, we have $sr_0=s$. As $V=Z(\mathfrak u)$ is abelian,
\[
\begin{aligned}
        t_i\exp(\lambda_i u)t_i^{-1}
        &=\exp\bigl(\rho_{\mathfrak u}(t_i)(\lambda_i u)\bigr)\\
        &=\exp(\lambda_i u)
          \exp\bigl(\rho_{\mathfrak u}(t_i)(\lambda_i u)-\lambda_i u\bigr).
\end{aligned}
\]
The left-hand side converges to $r_0$ and the right-hand side to $r_0\exp(u')$. Hence $r_0\exp(u')=r_0$, and multiplying by $s$ on the left gives $s\exp(u')=s$.
For each $\mu\in k$, repeating the argument with $\mu u$ gives $s\exp(\mu u')=s$. Thus $u'\in\mathfrak F_R(s)\cap V=W$.

Choose $j_0$ such that some $m_j$ differs from $m_{j_0}$. Then
\[
        u'-m_{j_0}u=\sum_{j=1}^r(m_j-m_{j_0})u_j
\]
is a nonzero element of $W$ with fewer than $r$ nonzero weight components, contradicting minimality of $r$. Therefore $W\cap V_\alpha\neq\{0\}$ for some $\alpha\in\Omega$. If $0\neq w\in W\cap V_\alpha$, the line $kw$ is the required $T$-invariant subspace.
\end{proof}

\subsection{The solvable-by-compact case}
\label{solvable-by-compact case}

\begin{definition}
Let $\mathbf G$ be a connected $k$-algebraic group, let
\[
        \mathbf G=\mathbf L\ltimes R_u(\mathbf G)
\]
be a Levi decomposition defined over $k$, and let $\mathbf T\leq\mathbf L$ be a maximal $k$-split torus.  We call $G=\mathbf G(k)$ \emph{algebraic solvable-by-compact} if $L=\mathbf L(k)$ centralises $T=\mathbf T(k)$.
\end{definition}

In this situation $\mathbf T$ is the unique maximal $k$-split torus of $\mathbf L$, and $L/T$ is compact \cite[Chapter~I, Proposition~2.3.6]{margulis1991discrete}. Thus $T\ltimes U$ is a closed solvable normal subgroup of $G$, and
\[
        G/(T\ltimes U)\cong L/T
\]
is compact.

\begin{lemma}\label{lem:levi-invariant-fixed-space}
Let $G$ be algebraic solvable-by-compact, with Levi factor $L$, and let $s\in w(G)$. If $\mathfrak F_R(s)\neq\{0\}$, then $\mathfrak F_R(s)\cap V$ contains a nonzero $L$-invariant $k$-subspace.
\end{lemma}

\begin{proof}
By Lemma~\ref{lem:torus-invariant-fixed-space}, there is a $T$-weight $\alpha$ in $V$ such that
\[
        W=\mathfrak F_R(s)\cap V_\alpha
\]
is a nonzero $k$-subspace. Since $L$ centralises $T$, the weight space $V_\alpha$ is $L$-invariant. Set
\[
        L_W=\{\ell\in L:\rho_{\mathfrak u}(\ell)W=W\}.
\]
The torus $T$ acts by scalars on $V_\alpha$, so $T\leq L_W$.

If $L_W$ is open, then $L_W/T$ is an open subgroup of the compact group $L/T$ and has finite index. Moreover, $L_W$ is the group of $k$-points of the algebraic stabiliser of $W$ in $\mathbf L$. Since $\mathbf L$ is connected and $L$ is Zariski dense in $\mathbf L$, this stabiliser cannot be proper and have finite index in $L$. Thus $L_W=L$.

Suppose instead that $L_W$ is not open. Choose a sequence $(\ell_n)_{n\geq1}$ in $L\setminus L_W$ with $\ell_n\to1_G$, and fix a complement
\[
        V_\alpha=W\oplus Y.
\]
Fix a finite basis of $W$. For every $n$, at least one basis vector is mapped outside $W$ by $\rho_{\mathfrak u}(\ell_n)$. Passing to a subsequence, choose a single basis vector $w$ for which this holds for every $n$, and write
\[
        \rho_{\mathfrak u}(\ell_n)w=w_n+y_n,
        \qquad w_n\in W,\quad 0\neq y_n\in Y.
\]
Then $w_n\to w$ and $y_n\to0$. Choose $\gamma_n\in k^\times$ with $\|\gamma_n y_n\|=1$. After passing to a subsequence, compactness of the unit sphere gives
\[
        \gamma_n y_n\longrightarrow y\in Y,
        \qquad \|y\|=1.
\]

Fix $\lambda\in k$. Passing to a common subnet, indexed by $i$, in $w(G)\times w(G)$, we get that
\[
        \exp(\lambda\gamma_i w)\longrightarrow r_1,
        \qquad
        \exp(\lambda\gamma_i w_i)\longrightarrow r_2.
\]
Both nets lie in $F_R(s)$, so $sr_1=sr_2=s$. Since $W,Y\subseteq V$,
\[
        \ell_i\exp(\lambda\gamma_i w)\ell_i^{-1}
        =\exp(\lambda\gamma_i w_i)\exp(\lambda\gamma_i y_i).
\]
The left-hand side converges to $r_1$, and the right-hand side to $r_2\exp(\lambda y)$. Consequently,
\[
        r_1=r_2\exp(\lambda y),
        \qquad s\exp(\lambda y)=s.
\]
Since $\lambda$ was arbitrary, $y\in\mathfrak F_R(s)\cap V_\alpha=W$, contradicting $0\neq y\in Y$. Thus $L_W=L$, and $W$ is $L$-invariant.
\end{proof}

\begin{cor}\label{cor:normal-unipotent-fixed}
Let $G$ be algebraic solvable-by-compact and let $s\in w(G)$. If $\mathfrak F_R(s)\neq\{0\}$, then there is a nontrivial connected normal unipotent $k$-subgroup $\mathbf N\lhd\mathbf G$ such that
\[
        \mathbf N(k)\leq F_R(s).
\]
\end{cor}

\begin{proof}
Let $0\neq W\subseteq V$ be the $L$-invariant $k$-subspace supplied by Lemma~\ref{lem:levi-invariant-fixed-space}. Let $\mathbf W$ be the corresponding vector subgroup of $\Lie(\mathbf U)$, and put $\mathbf N=\exp(\mathbf W)$. Since $W\subseteq Z(\mathfrak u)$, the exponential map is a group isomorphism on $\mathbf W$. Thus $\mathbf N$ is a nontrivial connected unipotent $k$-subgroup and
\[
        N=\mathbf N(k)=\exp(W).
\]
The stabiliser of $W$ in $\mathbf L$ is Zariski closed and contains the Zariski-dense subgroup $L=\mathbf L(k)$. Hence $\mathbf L$ normalises $\mathbf N$. Also, $\mathbf N\subseteq Z(\mathbf U)$, so $\mathbf U$ centralises $\mathbf N$. Since $\mathbf G=\mathbf L\ltimes\mathbf U$, we obtain $\mathbf N\lhd\mathbf G$. Finally, the definition of $\mathfrak F_R(s)$ gives $N\leq F_R(s)$.
\end{proof}

\begin{lemma}\label{lem:unipotent-part-kernel}
Let $G=L\ltimes U$ be algebraic solvable-by-compact, and recall that $C=\ker\rho_{\mathfrak u}$ and $V=Z(\mathfrak u)$. Set
\[
    L_0=C\cap L=\ker(\rho_{\mathfrak u}|_L),
    \qquad Z_L=Z(L).
\]
Then
\[
    C=L_0\times Z(U)=L_0\times\exp(V).
\]
Moreover, $L_0/(L_0\cap Z_L)$ is compact.
\end{lemma}

\begin{proof}
The algebraic image of $\mathbf L$ under $\rho_{\mathfrak u}$
is reductive, and the image of $\mathbf U$ is a normal unipotent
subgroup of the image of $\mathbf G$. Their intersection is a normal
unipotent subgroup of the reductive image of $\mathbf L$, and is
therefore trivial in characteristic zero. Thus, if $\ell u\in C$
with $\ell\in L$ and $u\in U$, then both images are trivial, so
\[
    \ell\in L_0,
    \qquad u\in C\cap U=Z(U)=\exp(V).
\]
Conversely, both $L_0$ and $Z(U)$ lie in $C$, commute, and have trivial intersection. This proves the direct-product decomposition.

Since $T\leq Z_L$ and $L/T$ is compact, $L/Z_L$ is compact. The algebraic quotient map by $Z(\mathbf L)$ maps $L_0$ to a closed subgroup, by \cite[Proposition~3.2]{bader2023homomorphic}. Therefore
\[
    L_0/(L_0\cap Z_L)\cong L_0Z_L/Z_L
\]
is a closed subgroup of the compact group $L/Z_L$, and is compact.
\end{proof}

\begin{lemma}\label{lem:kernel-boundary-fixed}
Let $G$ be algebraic solvable-by-compact, and let $s\in\overline C\subseteq w(G)$. Then either
\[
        s\in C\overline{Z(G)},
\]
or
\[
        \mathfrak F_R(s)\neq\{0\}.
\]
\end{lemma}

\begin{proof}
Write $G=L\ltimes U$ and put $\mathfrak l=\Lie(\mathbf L)(k)$. Since $\mathbf L$ is reductive and $k$ has characteristic zero, choose an $L$-invariant decomposition
\[
        V=V^L\oplus W,
\]
where $V^L$ is the subspace fixed by every element of $L$. Assume that $s\notin C\overline{Z(G)}$, and choose a net $(g_i)_{i\in\mathcal I}$ in $C$ converging to $s$.

By Lemma~\ref{lem:unipotent-part-kernel}, passing to a subnet, we may write
\[
        g_i=a_i\exp(w_i)z_i,
        \qquad a_i\longrightarrow a\in L_0,
        \quad w_i\in W,
        \quad z_i\in Z(G).
\]
The factors in $L_0\cap Z(L)$ and $\exp(V^L)$ commute with both $L$ and $U$, and hence we absorb them into $z_i\in Z(G)$.

If $(w_i)$ has a bounded subnet, then, after passing to further subnets,
\[
        w_i\longrightarrow w\in W,
        \qquad z_i\longrightarrow\sigma\in\overline{Z(G)}.
\]
Ellis--Lawson joint continuity gives
\[
        s=(a\exp(w))\sigma\in C\overline{Z(G)},
\]
a contradiction. Thus $\|w_i\|\to\infty$. Normalising $w_i$ and passing to a subnet, choose $\varepsilon_i\in k^\times$ such that
\[
        \varepsilon_i\longrightarrow0,
        \qquad \varepsilon_i w_i\longrightarrow y\in W\setminus\{0\}.
\]
Since $W^L=\{0\}$ and $\mathbf L$ is connected, there is $Y\in\mathfrak l$ for which
\[
        v\coloneqq [Y,y]\neq0.
\]
Fix $\lambda\in k$. Let $\exp_L$ denote the local analytic
exponential map of $L$. For all sufficiently large $i$, put
\[
        t_i=\exp_L(-\lambda\varepsilon_iY).
\]
Then $t_i\to1_G$, and
\[
    \bigl(\rho_{\mathfrak u}(t_i^{-1})-I\bigr)w_i
    =\lambda[Y,\varepsilon_iw_i]
      +O\bigl(|\varepsilon_i|^2\|w_i\|\bigr)
    \longrightarrow\lambda v.
\]
Set
\[
        c_i=g_i^{-1}t_i^{-1}g_it_i.
\]
Put $b_i=a_i^{-1}t_i^{-1}a_it_i$. Since $L_0$ is normal in $L$,
we have $b_i\in L_0$, so $b_i$ centralises $U$. Cancelling the central factor $z_i$, we obtain
\[
\begin{aligned}
c_i
&=\exp(-w_i)a_i^{-1}t_i^{-1}a_i\exp(w_i)t_i\\
&=\exp(-w_i)b_i\,t_i^{-1}\exp(w_i)t_i\\
&=b_i\exp(-w_i)
  \exp\bigl(\rho_{\mathfrak u}(t_i^{-1})w_i\bigr)\\
&=b_i\exp\bigl((\rho_{\mathfrak u}(t_i^{-1})-I)w_i\bigr),
\end{aligned}
\]
where the last equality uses that $V$ is abelian and $L$-invariant.
The first factor tends to $1_G$, and hence $c_i\to\exp(\lambda v)$. Finally, $g_ic_i=t_i^{-1}g_it_i$, so Ellis--Lawson joint continuity gives
\[
        s\exp(\lambda v)
        =\lim_i(g_ic_i)
        =\lim_i(t_i^{-1}g_it_i)
        =s.
\]
Since $\lambda$ was arbitrary and $v\neq0$, we obtain $0\neq v\in\mathfrak F_R(s)$.
\end{proof}

\begin{cor}[Strong WAP Mautner property in the solvable-by-compact case]
\label{cor:cbs-strong-mautner}
Let $G$ be algebraic solvable-by-compact. For every
\[
        s\in w(G)\setminus G\overline{Z(G)},
\]
there is a nontrivial connected normal unipotent $k$-subgroup $\mathbf N\lhd\mathbf G$ such that
\[
        \mathbf N(k)\leq F_R(s).
\]
In particular, if $Z(G)$ is compact, this conclusion holds for every $s\in w(G)\setminus G$.
\end{cor}

\begin{proof}
Fix $s\notin G\overline{Z(G)}$. By \cite[Proposition~3.2]{bader2023homomorphic}, the image $\rho_{\mathfrak u}(G)$ is closed in $\GL(\mathfrak u)$, and the representation induces a topological isomorphism
\[
        G/C\cong\rho_{\mathfrak u}(G).
\]
Let
\[
        q_{\rho_{\mathfrak u}}\colon w(G)\longrightarrow w(\rho_{\mathfrak u}(G))
\]
be the induced morphism, and choose a net $(g_i)_{i\in\mathcal I}$ in $G$ converging to $s$.

Suppose first that $q_{\rho_{\mathfrak u}}(s)\notin\rho_{\mathfrak u}(G)$. Then $(\rho_{\mathfrak u}(g_i))$ escapes every compact subset of $\rho_{\mathfrak u}(G)$. Since this image is closed in $\GL(\mathfrak u)$, the net also escapes every compact subset of $\GL(\mathfrak u)$. Lemma~\ref{lem:adjoint-divergence-fixed} gives $\mathfrak F_R(s)\neq\{0\}$.

Suppose instead that $q_{\rho_{\mathfrak u}}(s)=\rho_{\mathfrak u}(g)$ for some $g\in G$. Under the identification $G/C\cong\rho_{\mathfrak u}(G)$, Lemma~\ref{lem:lift-group-point} gives
\[
        s=g\sigma,
        \qquad \sigma\in\overline C.
\]
Left multiplication by $g$ does not change right stabilisers, so
\[
        F_R(s)=F_R(\sigma),
        \qquad \mathfrak F_R(s)=\mathfrak F_R(\sigma).
\]
If $\sigma\in C\overline{Z(G)}$, then $s\in gC\overline{Z(G)}\subseteq G\overline{Z(G)}$, contrary to the choice of $s$. Lemma~\ref{lem:kernel-boundary-fixed} therefore gives $\mathfrak F_R(\sigma)\neq\{0\}$ and hence $\mathfrak F_R(s)\neq\{0\}$.

In either case, Corollary~\ref{cor:normal-unipotent-fixed} supplies a nontrivial connected normal unipotent $k$-subgroup $\mathbf N\lhd\mathbf G$ with $\mathbf N(k)\leq F_R(s)$.

Finally, if $Z(G)$ is compact, it is closed in $w(G)$, so $G\overline{Z(G)}=GZ(G)=G$. This proves the last assertion.
\end{proof}

\begin{theorem}\label{thm:cbs-cc}
Every algebraic solvable-by-compact group is compactification-centric.
\end{theorem}

\begin{proof}
We argue by induction on $\dim\mathbf G$. If $\dim\mathbf G\leq1$, then $G$ is abelian, and separate continuity implies that every element of $w(G)$ commutes with $G$. Thus $G$ is compactification-centric.

Let $s\in w(G)$. If $s\in G\overline{Z(G)}$, write $s=g\sigma$ with $g\in G$ and $\sigma\in\overline{Z(G)}$. Separate continuity implies that $\sigma$ commutes with $G$, and therefore
\[
        sG=g\sigma G=gG\sigma=Gg\sigma=Gs.
\]
Now suppose that $s\notin G\overline{Z(G)}$. Corollary~\ref{cor:cbs-strong-mautner} gives a nontrivial connected normal unipotent $k$-subgroup $\mathbf N\lhd\mathbf G$ such that $sN=s$, where $N=\mathbf N(k)$. Since $k$ has characteristic zero, $\mathbf N$ is $k$-split. Thus, by \cite[Propositions~3.2 and~3.6]{bader2023homomorphic},
we have a topological isomorphism
\[
        G/N\cong(\mathbf G/\mathbf N)(k).
\]
This quotient is again algebraic solvable-by-compact: the image of $\mathbf L$ is an isomorphic Levi factor and still centralises its maximal $k$-split torus. It has smaller dimension because $\mathbf N$ is connected and nontrivial. By induction, $G/N$ is compactification-centric, and Lemma~\ref{lem:fixed-normal-quotient} gives $sG=Gs$.
\end{proof}

\subsection{The general case}

We now pass from minimal parabolic subgroups to the full algebraic group.  The first result records the strong Mautner statement.

\begin{theorem}[Strong WAP Mautner property]\label{thm:strong-wap-mautner}
Let $\mathbf G$ be a connected affine $k$-algebraic group and put $G=\mathbf G(k)$.  Assume that $G$ is noncompact and that $Z(G)$ is compact.  Then, for every
\[
        s\in w(G)\setminus G,
\]
the subgroup $F_R(s)$ contains $\mathbf N(k)$ for some nontrivial connected unipotent $k$-subgroup $\mathbf N\leq\mathbf G$. In particular, $G$ has the WAP Mautner property.
\end{theorem}

\begin{proof}
Fix $s\in w(G)\setminus G$, and let $P=\mathbf P(k)$ for a minimal
$k$-parabolic subgroup $\mathbf P$.  Since $P$ is cocompact in $G$,
Remark~\ref{rem:cocompact-decomposition} gives
\[
        s=g\sigma
        \qquad
        \text{with }g\in G,
        \quad
        \sigma\in\overline P\subseteq w(G).
\]
The subgroup $P$ is noncompact, algebraic solvable-by-compact, and $Z(P)=Z(G)$ by \cite[Lemma~3.8]{bader2023homomorphic}.  In particular, $Z(P)$ is
compact.

Regard $\overline P$ as a compact semitopological semigroup compactification
of $P$.  The universal property gives a continuous surjective semigroup
homomorphism
\[
        \theta\colon w(P)\longrightarrow\overline P.
\]
Choose $\widetilde\sigma\in w(P)$ with
$\theta(\widetilde\sigma)=\sigma$.  Since $s\notin G$, one has
$\sigma\notin P$, and therefore $\widetilde\sigma\notin P$.
Corollary~\ref{cor:cbs-strong-mautner}, applied in $w(P)$, gives a
nontrivial connected normal unipotent $k$-subgroup $\mathbf N\lhd\mathbf P$
such that
\[
        \mathbf N(k)\leq F_R(\widetilde\sigma).
\]
Applying $\theta$ yields
\[
        \mathbf N(k)\leq F_R(\sigma)=F_R(g\sigma)=F_R(s).
\]
\end{proof}
We recall the following well-known consequence of the Bruhat
decomposition.

\begin{lemma}\label{lem:generated-by-minimal-parabolics}
Let $\mathbf G$ be a connected affine $k$-algebraic group and
let $\mathbf P$ be a minimal $k$-parabolic subgroup. Put
$G=\mathbf G(k)$ and $P=\mathbf P(k)$. Then
\[
        G=\left\langle gPg^{-1}:g\in G\right\rangle.
\]
\end{lemma}

\begin{proof}
Put $H=\langle gPg^{-1}:g\in G\rangle$ and let
$T=\mathbf T(k)$ for a maximal $k$-split torus
$\mathbf T\leq\mathbf P$.
Every unipotent radical of a $k$-parabolic subgroup is contained
in a minimal $k$-parabolic subgroup. By conjugacy of these
subgroups, $H$ contains $G^+$, the subgroup generated by the
$k$-points of those radicals.
Thus \cite[Remark~3.2]{conradTitsSystemsRootGroups} gives
\[
        G=HZ_G(T)=H,
\]
since $Z_G(T)\subseteq P\subseteq H$.
\end{proof}

We are now ready to prove Theorem~\ref{thm:compactification-centric}.
\begin{proofof}{Theorem~\ref{thm:compactification-centric}}
Fix $s\in w(G)$ and let $P=\mathbf P(k)$ for a minimal $k$-parabolic subgroup $\mathbf P\leq\mathbf G$.  Let $Q=gPg^{-1}$ be any conjugate of $P$.  The subgroup $Q$ is cocompact and algebraic solvable-by-compact.  Remark~\ref{rem:cocompact-decomposition} provides
\[
        s=\sigma x=z\tau
\]
with $x,z\in G$ and $\sigma,\tau\in\overline Q$.  By Theorem~\ref{thm:cbs-cc}, the group $Q$ is compactification-centric.  Thus, for every $q\in Q$, there are $q_1,q_2\in Q$ such that
\[
        q\sigma=\sigma q_1,
        \qquad
        \tau q=q_2\tau.
\]
Consequently,
\[
        qs=q\sigma x=\sigma q_1x=sx^{-1}q_1x\in sG
\]
and
\[
        sq=z\tau q=zq_2\tau=zq_2z^{-1}s\in Gs.
\]
The two required orbit inclusions therefore hold for every element of every conjugate of $P$.

By Lemma~\ref{lem:generated-by-minimal-parabolics}, those conjugates generate $G$.  These inclusions are stable under products. Indeed, if $bs=sh$ and $as=sk$ for some $h,k\in G$, then
\[
        (ab)s=a(sh)=(as)h=skh\in sG.
\]
The analogous calculation proves stability of the opposite inclusion. Hence
\[
        Gs\subseteq sG,
        \qquad
        sG\subseteq Gs.
\]
\end{proofof}

\begin{remark}\label{ex:disconnected-factor-swap}
Algebraic connectedness cannot be omitted from
Theorem~\ref{thm:compactification-centric}.
Indeed, let $H=\mathrm{SL}_2(k)$ and put
\[
        G=(H\times H)\rtimes C_2,
\]
where the nontrivial element $\tau\in C_2$ exchanges
the two factors. On $L^2(H)\oplus L^2(H)$, consider
the unitary representation
\[
        \pi(h_1,h_2)=\lambda_H(h_1)\oplus\lambda_H(h_2),
        \qquad
        \pi(\tau)(\xi,\eta)=(\eta,\xi),
\]
where $\lambda_H$ is the left regular representation.
Since its matrix coefficients vanish at infinity,
a net $h_i\to\infty$ in $H$ gives
\[
        \pi(h_i,1_H)\longrightarrow P\coloneqq0\oplus I
        \qquad\text{in the weak operator topology}.
\]
Thus $P\in J(S_\pi)$, but
\[
        \pi(\tau)P\pi(\tau)^{-1}=I\oplus0\neq P.
\]
By Proposition~\ref{prop:cor_of_compactification_centric},
this rules out compactification-centricity of $G$.
\end{remark}

\section{WAP compactifications and harmonic analysis}
\label{sec:harmonic-applications}

Throughout this section, $k$ is a local field of characteristic zero,
$\mathbf G$ is a connected affine $k$-algebraic group, and
$G=\mathbf G(k)$. All unitary representations are strongly continuous.
We retain the notation $\mathbf U=R_u(\mathbf G)$,
$\mathfrak u=\Lie(\mathbf U)(k)$, $V=Z(\mathfrak u)$, and
$\rho_{\mathfrak u}\colon G\to\GL(\mathfrak u)$ from the preceding section.

We now use compactification-centricity to relate the structure of $w(G)$ to the harmonic analysis of $G$. The central step is a \textit{decomposition theorem} for sealed algebraic groups: $w(G)$ is the disjoint union of the locally compact groups $Ge$, indexed by its idempotents, and its idempotents admit no infinite strictly decreasing sequence. We prove both assertions by induction on the dimension of the unipotent radical, using the semisimple WAP theorem of Bader and Gelander
\cite[Theorems~7.13 and~7.14]{MR3692904} for the reductive base case.

This structure has two consequences. Regular representations of the groups $Ge$ produce matrix coefficients separating the points of $w(G)$, while the idempotents yield an exhaustive orthogonal decomposition in every unitary representation. These arguments prove that sealedness implies both the Eberlein property and the representation decomposition below. We then prove the converse implications using quotients with noncompact centre.

We prove the following strengthened form of Theorem~\ref{mthm:main}.

\begin{theorem}\label{thm:harmonic-equivalences}
The following conditions are equivalent:
\begin{enumerate}
\item For every closed normal subgroup $N\lhd G$, the centre of $G/N$
is compact.
\item The group $G$ is sealed.
\item Every unitary representation of $G$ is a Hilbert direct sum of
representations $\pi_j$ such that, for each $j$, either $\pi_j$ is
finite-dimensional or its induced representation of $G/\ker\pi_j$
is mixing. Each quotient carries the quotient topology.
\item The group $G$ is Eberlein: $E(G)=\operatorname{WAP}(G)$.
\end{enumerate}
\end{theorem}

\subsection{The structure of the WAP compactification}

\begin{theorem}[Decomposition theorem]\label{thm:wap-decomposition}
If $G$ is sealed, then
\[
        w(G)=\bigsqcup_{e\in J(w(G))}Ge.
\]
Moreover, the ordered set $J(w(G))$, with $e\leq f$ if $ef=e$, has no
infinite strictly decreasing sequence.
\end{theorem}

We first record two useful facts about the group orbits.
For an idempotent $e$ of a semigroup, $H_e$ denotes its maximal subgroup,
whose identity is $e$.

\begin{lemma}\label{lem:sealed-orbit-group}
Suppose that $G$ is sealed, and let $S$ be a compact semitopological
semigroup compactification of $G$, with canonical map $\varphi\colon G\to S$.
For every $e\in J(S)$,
\[
        H_e=\varphi(G)e.
\]
Moreover, the homomorphism $g\mapsto\varphi(g)e$ induces a topological
isomorphism
\[
        G/F_R(e)\cong\varphi(G)e.
\]
In particular, $\varphi(G)e$ is locally compact and is open in $Se$.
\end{lemma}

\begin{proof}
By Theorem~\ref{thm:compactification-centric} and Proposition~\ref{prop:cor_of_compactification_centric}, $e$ is central and $\varphi(G)e$ is a subgroup of $H_e$. The group of units $H_e$ of the compact semitopological semigroup $Se$ is a topological group in its subspace topology, by the Ellis--Lawson theorem \cite[II.4.11]{ruppert2006compact}. Sealedness makes $\varphi(G)e$ closed in $H_e$. Since $Se=\overline{\varphi(G)e}$, it is also dense in $H_e$. Thus $H_e=\varphi(G)e$.

Since $G/F_R(e)$ is Polish and sealed, the induced continuous bijective homomorphism onto $\varphi(G)e$ is a homeomorphism \cite[Lemma~4.12]{bader2023homomorphic}. Finally, a locally compact dense subspace of a compact Hausdorff space is open.
\end{proof}

\begin{lemma}\label{lem:fixed-locus-quotient}
Let $N\lhd G$ be a closed normal subgroup, and let $q_N\colon w(G)\to w(G/N)$ be the
canonical morphism. Its restriction to
\[
        X_N\coloneqq\{s\in w(G):sN=s\}
\]
is injective. If $s\in X_N$ and $q_N(s)$ belongs to the orbit of an idempotent in $w(G/N)$, then $s$ belongs to the orbit of an idempotent
in $w(G)$.
\end{lemma}

\begin{proof}
Put $M=M(\overline N)$ and let $e_M$ be the identity of the compact group $M$. For $s\in X_N$, separate continuity gives $sM=\{s\}$. 
Hence for $s_1,s_2\in X_N$ if $q_N(s_1)=q_N(s_2)$ then
\[
\{s_1\}=s_1M=s_2M=\{s_2\}.
\]
Now, suppose that $q_N(s)=(gN)v$ with $v\in J(w(G/N))$. The compact
subsemigroup $q_N^{-1}(\{v\})$ contains an idempotent $e$.
Then $q_N(s)=q_N(ge)$, so
\[
        \{s\}=sM=geM.
\]
In particular, $s=gee_M$. Both $e$ and $e_M$ are central idempotents, so
$ee_M$ is an idempotent, as required.
\end{proof}

We first show that the decomposition passes from a subgroup of finite index.
\begin{lemma}\label{lem:finite-index-decomposition}
Let $S$ be a compact semitopological semigroup with a dense subgroup $G$, and suppose that all idempotents of $S$ are central.
Let $H\leq G$ have finite index. If
\[
        \overline H=\bigcup_{e\in E}He
\]
for a finite set $E\subseteq J(\overline H)$, then $J(S)=E$ and
\[
        S=\bigcup_{e\in E}Ge.
\]
\end{lemma}

\begin{proof}
Choose coset representatives $a_1,\ldots,a_m$ such that $G=\bigcup_{j=1}^m a_jH$. Then
\[
        S=\bigcup_{j=1}^m a_j\overline H
         =\bigcup_{e\in E}Ge.
\]
Since $e$ is central, $Ge$ is a group with identity $e$, and therefore contains
no other idempotent. Hence $J(S)=E$.
\end{proof}

We now adapt the semisimple WAP theorem of Bader and Gelander
to the reductive case.

\begin{lemma}\label{lem:reductive-closure}
Let $\mathbf A\leq\mathbf G$ be a connected reductive
$k$-subgroup such that $Z(\mathbf A)(k)$ is compact.
Put $A=\mathbf A(k)$. Then $J(\overline A)$ is finite and
\[
        \overline A=\bigcup_{e\in J(\overline A)}Ae.
\]
\end{lemma}

\begin{proof}
By the structure theory of connected reductive groups
\cite[§§14 and 22]{borel2012linear}, there is a central
$k$-isogeny
\[
        Z(\mathbf A)^\circ\times
        \mathbf D_1\times\cdots\times\mathbf D_m
        \longrightarrow\mathbf A,
\]
where the $\mathbf D_j$ are simply connected almost
$k$-simple groups. Grouping the compact factors together,
the induced map on $k$-points has the form
\[
        f\colon D=K\times D_1\times\cdots\times D_r
        \longrightarrow A,
\]
where $K$ is compact and each $D_j$ is noncompact.
Its kernel is finite, and its image $H=f(D)$ has finite index in $A$.
%, by the finiteness of the relevant Galois cohomology set \cite[Theorem~6.14]{platonov1993algebraic}.

The semisimple WAP theorem \cite[Theorems~7.13 and~7.14]{MR3692904} gives
\[
        w(D_1\times\cdots\times D_r)
        \cong\prod_{j=1}^r(D_j\cup\{0_j\}),
\]
where $D_j\cup\{0_j\}$ is the one-point compactification of $D_j$, with $0_j$ an absorbing zero. Compactness of $K$ therefore gives
\[
        w(D)\cong K\times\prod_{j=1}^r(D_j\cup\{0_j\}).
\]
Its idempotents are the $2^r$ tuples whose $K$-coordinate is $1_K$ and whose $j$th coordinate is either $1_{D_j}$ or $0_j$. Every point of $w(D)$ belongs to the $D$-orbit of one of these idempotents.

The map $f$ extends to a continuous surjective semigroup homomorphism
\[
        \widetilde f\colon w(D)\longrightarrow\overline H,
\]
where $\overline H$ denotes the closure of $H=f(D)$ in $w(G)$.

Thus, putting $E=\widetilde f(J(w(D)))$, we obtain
\[
        \overline H=\bigcup_{e\in E}He.
\]
Since $H$ has finite index in $A$ and all idempotents
of $w(G)$ are central, Lemma~\ref{lem:finite-index-decomposition}
gives
\[
        \overline A=\bigcup_{e\in E}Ae,
        \qquad J(\overline A)=E.
\]
\end{proof}

Fix a Levi decomposition $\mathbf G=\mathbf L\ltimes\mathbf U$
and put $L=\mathbf L(k)$. The next lemma supplies the induction step.

Recall that $\rho_{\mathfrak u}$ denotes the action of $G$ on $\mathfrak u$ induced by conjugation on $U$.
We use the same notation for the algebraic representation
of $\mathbf G$ on $\Lie(\mathbf U)$.

\begin{lemma}\label{lem:sealed-boundary-alternative}
Assume that $G$ is sealed, and put
\[
        \mathbf L_0=\ker(\rho_{\mathfrak u}|_{\mathbf L}),
        \qquad L_0=\mathbf L_0(k).
\]
For every $s\in w(G)$, at least one of the following holds:
\begin{enumerate}
\item $s\in G\overline{L_0}$;
\item there is a nontrivial connected normal unipotent
$k$-subgroup $\mathbf N\lhd\mathbf G$, contained in
$Z(\mathbf U)$, such that $s\mathbf N(k)=s$.
\end{enumerate}
Moreover, $\mathbf L_0^\circ$ is reductive and
$Z(\mathbf L_0^\circ)(k)$ is compact.
\end{lemma}

\begin{proof}
Since $G$ is sealed, $Z(L)$ is compact and
$\mathfrak u^L=\{0\}$ by \cite[Theorem~8.1]{bader2023homomorphic}. In particular, $V^L=\{0\}$. Moreover, $\mathbf L_0^\circ$ is a connected normal subgroup of the reductive group $\mathbf L$, so it is reductive, and its centre has compact $k$-points by \cite[Lemma~8.8]{bader2023homomorphic}.

Set $C=\ker\rho_{\mathfrak u}=C_G(U)$. 
Note that 
\[C=L_0\times Z(U)=L_0\times\exp(V)\] 
since $\rho_{\mathfrak u}(\mathbf L)\cap \rho_{\mathfrak u}(\mathbf U)=\{I\}$. Thus $\ell u\in C$ if and only if $\ell\in L_0$ and
$u\in Z(U)$, and the product is direct since $L_0$
centralises $U$ and $L_0\cap Z(U)=\{1_G\}$.

We next observe that $\mathfrak F_R(s)\neq\{0\}$ implies (2).
Indeed, Corollary~\ref{cor:fixed-central-vector} gives
\[
        W=\mathfrak F_R(s)\cap V\neq\{0\}.
\]
Normality of $F_R(s)$ makes $W$ invariant under $L$.
Let $\mathbf N$ be the connected unipotent $k$-subgroup
obtained by exponentiating the vector subgroup corresponding
to $W$. Then $\mathbf N\subseteq Z(\mathbf U)$ and
\[
        \mathbf N(k)=\exp(W)\leq F_R(s).
\]
Zariski density of $L$ implies that $\mathbf L$ normalises
$\mathbf N$. Since $\mathbf U$ centralises $\mathbf N$,
we have $\mathbf N\lhd\mathbf G$, proving (2).

Consider now the image of $s$ under
$q_C\colon w(G)\to w(G/C)$. The representation
$\rho_{\mathfrak u}$ induces a topological isomorphism
\[
        G/C\cong\rho_{\mathfrak u}(G),
\]
and its image is closed in $\GL(\mathfrak u)$
\cite[Proposition~3.2]{bader2023homomorphic}.
If $q_C(s)\notin G/C$, then for any net $g_i\in G$
converging to $s$, the net $\rho_{\mathfrak u}(g_i)$
escapes every compact subset of $\GL(\mathfrak u)$.
Lemma~\ref{lem:adjoint-divergence-fixed} then gives
$\mathfrak F_R(s)\neq\{0\}$, and hence (2).

Otherwise, by Lemma~\ref{lem:lift-group-point}
$s=g\sigma$ for some $g\in G$ and $\sigma\in\overline C$.
Choose a net
\[
        h_i=a_i\exp(v_i)\longrightarrow\sigma,
        \qquad a_i\in L_0,\quad v_i\in V.
\]
If $(v_i)$ has a bounded subnet, we can assume that $v_i\to v$. Joint continuity gives
\[
        a_i=h_i\exp(-v_i)
        \longrightarrow \alpha\coloneqq\sigma\exp(-v)
        \in\overline{L_0}.
\]
Since $L_0$ centralises $V$, so does $\overline{L_0}$.
Consequently,
\[
        s=g \alpha\exp(v)=g\exp(v) \alpha\in G\overline{L_0},
\]
which is (1).

It remains to consider the case $\|v_i\|\to\infty$.
After normalisation and passage to a subnet, choose
$\varepsilon_i\in k^\times$ such that
\[
        \varepsilon_i\longrightarrow0,
        \qquad \varepsilon_i v_i\longrightarrow
        v_0\in V\setminus\{0\}.
\]
We shall find small elements of $L$ which commute with
the $a_i$ but act nontrivially on $v_0$.

Write
\[
        \mathbf L
        =Z(\mathbf L)^\circ\mathbf S_1\cdots\mathbf S_m,
\]
where the $\mathbf S_j$ are the almost $k$-simple normal
factors of the derived group. Since $\mathbf L_0$ is normal,
each $\mathbf S_j$ is either contained in $\mathbf L_0^\circ$
or has finite intersection with $\mathbf L_0$.
In the latter case, $\mathbf S_j$ centralises $\mathbf L_0$:
for each $a\in\mathbf L_0$, the commutator map
$x\mapsto xax^{-1}a^{-1}$ maps the connected group
$\mathbf S_j$ into the finite group
$\mathbf S_j\cap\mathbf L_0$, so it is trivial.

Let $\mathbf L_1$ be the product of $Z(\mathbf L)^\circ$
and the factors not contained in $\mathbf L_0^\circ$.
Put $L_1=\mathbf L_1(k)$. Then $\mathbf L_1$ is connected and reductive, centralises
$\mathbf L_0$, and
\[
        \mathbf L=\mathbf L_0^\circ\mathbf L_1.
\]
Since $\mathbf L_0$ acts trivially on $V$ and $V^L=\{0\}$,
the nonzero vector $v_0$ is not fixed by $\mathbf L_1$.
Connectedness and characteristic zero therefore give
$Y\in\Lie(\mathbf L_1)(k)$ such that
\[
        v\coloneqq [Y,v_0]\neq0.
\]

Fix $\lambda\in k$. Using the local exponential map of $L$,
put, for all sufficiently large $i$,
\[
        t_i=\exp_L(-\lambda\varepsilon_iY).
\]
Then $t_i\in L_1$, $t_i\to1_G$, and $t_i$ commutes with $a_i$.
The first-order expansion gives
\[
\begin{aligned}
        (\rho_{\mathfrak u}(t_i^{-1})-I)v_i
        &=
        \lambda\,d\rho_{\mathfrak u}(Y)(\varepsilon_i v_i)
        +O(|\varepsilon_i|^2\|v_i\|)\\
        &\longrightarrow\lambda v,
\end{aligned}
\]
since $|\varepsilon_i|\|v_i\|$ is bounded.
Thus, putting $c_i=h_i^{-1}t_i^{-1}h_it_i$, we obtain
\[
\begin{aligned}
        c_i
        &=\exp(-v_i)t_i^{-1}\exp(v_i)t_i\\
        &=\exp\bigl((\rho_{\mathfrak u}(t_i^{-1})-I)v_i\bigr)
        \longrightarrow\exp(\lambda v).
\end{aligned}
\]
Here we used that $V$ is abelian. Finally,
$h_ic_i=t_i^{-1}h_it_i$, so Ellis--Lawson joint continuity
gives
\[
        \sigma\exp(\lambda v)
        =\lim_i h_ic_i
        =\lim_i t_i^{-1}h_it_i
        =\sigma.
\]
Since $\lambda$ was arbitrary,
$0\neq v\in\mathfrak F_R(\sigma)$.
Left multiplication by $g\in G$ does not change the right
stabiliser, so $\mathfrak F_R(\sigma)=\mathfrak F_R(s)$.
Hence (2) holds.
\end{proof}
\begin{proofof}{Theorem~\ref{thm:wap-decomposition}}
We prove both assertions by induction on $d=\dim\mathbf U$.
When $d=0$, $\mathbf G$ is reductive and $Z(G)$ is compact, so both
assertions follow from Lemma~\ref{lem:reductive-closure}.

Assume $d>0$ and take $s\in w(G)$. In the first case of
Lemma~\ref{lem:sealed-boundary-alternative}, we have
$s\in G\overline{L_0}$.
Applying Lemma~\ref{lem:reductive-closure} to
$\mathbf L_0^\circ$ and then
Lemma~\ref{lem:finite-index-decomposition} to the
finite-index subgroup $\mathbf L_0^\circ(k)\leq L_0$,
we obtain
\[
        \overline{L_0}
        =\bigcup_{e\in J(\overline{L_0})}L_0e,
        \qquad |J(\overline{L_0})|<\infty.
\]
Consequently,
\[
        s\in G\overline{L_0}
        =\bigcup_{e\in J(\overline{L_0})}Ge.
\]

In the second case, let $N=\mathbf N(k)$ be as in that lemma.
The quotient is topologically
\[
        G/N\cong(\mathbf G/\mathbf N)(k),
\]
because $\mathbf N$ is connected and unipotent in characteristic zero
\cite[Propositions~3.2 and~3.6]{bader2023homomorphic}.
It is sealed, and its unipotent radical is $\mathbf U/\mathbf N$, of
dimension less than $d$. The induction hypothesis and
Lemma~\ref{lem:fixed-locus-quotient} therefore place $s$ in an orbit $Ge$.

These orbits are disjoint. Indeed, each $Ge$ is a group with identity
$e$. More explicitly, if $e=gf$ for central idempotents $e,f$, then
$e=e^2=g^2f$; multiplying by $g^{-1}$ gives $f=gf=e$.
An intersection $Ge\cap Gf$ reduces to this situation.

To prove the second assertion, suppose that
$e_1>e_2>\cdots$ is a strictly decreasing sequence of idempotents.
If every $e_j$ belongs to $G\overline{L_0}$, the decomposition just
proved and uniqueness of the idempotent in a group orbit imply
$e_j\in J(\overline{L_0})$. This contradicts finiteness of that set.
Otherwise, for some $j$, the second case of
Lemma~\ref{lem:sealed-boundary-alternative} gives $e_jN=e_j$ with
$N=\mathbf N(k)$ as above. For $m\geq j$,
\[
        e_mN=e_me_jN=e_me_j=e_m.
\]
Lemma~\ref{lem:fixed-locus-quotient} maps this tail injectively into
$J(w(G/N))$, preserving its order. This contradicts the induction
hypothesis for $G/N$.
\end{proofof}

\subsection{The Eberlein property}
\label{subsec:sealed_and_eberlin}

Recall that $E(G)=\overline{B(G)}^{\|\cdot\|_\infty}$.
We first use the group orbits in $w(G)$ to construct enough matrix
coefficients to separate its points.

\begin{prop}\label{prop:sealed-implies-eberlein}
If $G$ is sealed, then $G$ is Eberlein.
\end{prop}

\begin{proof}
Put $S=w(G)$ and fix $e\in J(S)$. By
Lemma~\ref{lem:sealed-orbit-group}, the maximal subgroup
$H_e=Ge$ is the group of units of $Se$ and is locally
compact and open in $Se$.

By Theorem~\ref{thm:wap-decomposition}, every $x\in S$
has the form $x=gf$. Since $x\in Se$ if and only if
$xe=x$, equivalently $fe=f$, we have
\[
        Se=\bigsqcup_{f\leq e}Gf.
\]
Since $H_e=Ge$ and the orbits are disjoint, removing
the term $f=e$ gives
\[
        Se\setminus H_e=\bigsqcup_{f<e}Gf.
\]
Centrality of the idempotents gives
$(Gf)(Gh)\subseteq Gfh$. If $f<e$ and $h\leq e$, then
$fh\leq f<e$. Thus $Se\setminus H_e$ is a two-sided
ideal of $Se$.

Let $\lambda_{H_e}$ be the left regular representation
of $H_e$ on $L^2(H_e)$. Since $g\mapsto ge$ is a
continuous group homomorphism, it defines a unitary
representation of $G$ by
\[
        \pi_e(g)=\lambda_{H_e}(ge).
\]
We extend $\lambda_{H_e}$ to $Se$ by sending
$Se\setminus H_e$ to the zero operator. The ideal
property ensures that this extension preserves
multiplication. It is weak-operator continuous:
$H_e$ is open in $Se$, and the matrix coefficients
of $\lambda_{H_e}$ vanish at infinity, so their
extensions by zero are continuous on $Se$.
If $H_e$ is compact, it is closed and dense in $Se$,
so $Se=H_e$.

Composing this extension with $s\mapsto se$ gives a
continuous semigroup homomorphism
\[
        R_e\colon S\longrightarrow\mathcal B(L^2(H_e)),
        \qquad
        R_e(s)=
        \begin{cases}
        \lambda_{H_e}(se),&se\in H_e,\\
        0,&se\notin H_e,
        \end{cases}
\]
where the operator space carries the weak operator
topology. For $g\in G$, we have $R_e(g)=\pi_e(g)$.

The maps $R_e$ separate points of $S$. Indeed, write
$s=ge$ and $t=hf$, with $s\neq t$. If $e=f$, then
$s,t\in H_e$, and faithfulness of $\lambda_{H_e}$
gives $R_e(s)\neq R_e(t)$. If $e\neq f$, at least
one of $ef\neq e$ or $ef\neq f$ holds. In the first
case, $te\in G(ef)$ lies outside $H_e$, so
$R_e(t)=0$, whereas $R_e(s)$ is unitary.
The second case is symmetric, using $R_f$.

Distinct operators are separated by matrix
coefficients. Hence the continuous functions
\[
        s\longmapsto\langle R_e(s)\xi,\eta\rangle,
        \qquad e\in J(S),\quad \xi,\eta\in L^2(H_e),
\]
separate points of $S$. Their restrictions to $G$
are matrix coefficients of the representations
$\pi_e$, and therefore belong to $B(G)$.
The continuous extensions of $B(G)$ thus form a
unital self-adjoint subalgebra of $C(S)$ separating
points. By the Stone--Weierstrass theorem, this
algebra is uniformly dense in $C(S)$. Identifying
$C(S)$ with $\operatorname{WAP}(G)$ by restriction
to $G$, we conclude that
\[
        E(G)=\operatorname{WAP}(G).
\]
\end{proof}
\subsection{Decomposition of unitary representations}

\begin{lemma}\label{lem:operator-idempotents}
Let $\pi\colon G\to\mathcal U(\mathcal H)$ be a unitary representation.
Every idempotent in $S_\pi$ is an orthogonal projection onto a
$G$-invariant subspace.
\end{lemma}
\begin{proof}
Every element of $S_\pi$ is a contraction, and every
idempotent contraction on a Hilbert space is an orthogonal
projection. By compactification-centricity, $P\in J(S_\pi)$
commutes with $\pi(G)$, so its range is $G$-invariant.
\end{proof}

\begin{prop}\label{prop:sealed-representation-decomposition}
If $G$ is sealed, every unitary representation of $G$ has the
decomposition specified in Theorem~\ref{thm:harmonic-equivalences}(3).
\end{prop}

\begin{proof}
Let $\pi\colon G\to\mathcal U(\mathcal H)$ be a unitary
representation,
and put $S=S_\pi$. Theorem~\ref{thm:wap-decomposition} and the canonical
surjection $w(G)\to S$ imply that every element of $S$ has the form
$\pi(g)P$ with $P\in J(S)$. Moreover, $J(S)$ has no infinite strictly decreasing sequence:
lifting the terms of such a sequence to idempotents in $w(G)$
and taking successive products of the lifts would yield an
infinite strictly decreasing sequence in $J(w(G))$.

For $P\in J(S)$, define the closed invariant subspace
\[
        \mathcal K_P
        \coloneqq P\mathcal H\cap
        \left(\overline{\operatorname{span}}
              \{Q\mathcal H:Q\in J(S),\ Q<P\}\right)^\perp.
\]
These subspaces are mutually orthogonal. Indeed, for distinct $P,Q$,
at least one of $PQ<P$ or $PQ<Q$ holds. For
$\xi\in\mathcal K_P$ and $\eta\in\mathcal K_Q$, use
\[
        \langle\xi,\eta\rangle
        =\langle\xi,PQ\eta\rangle
        =\langle PQ\xi,\eta\rangle.
\]
The defining orthogonality gives zero in either case.

They also span $\mathcal H$. Otherwise, let $\mathcal R$ be the nonzero
orthogonal complement of their closed span. It is $G$-invariant, so
every element of $S$ preserves it. Since the identity belongs to $J(S)$,
there is, by the descending-chain condition, a minimal $P\in J(S)$
with $P\mathcal R\neq\{0\}$. Every $Q<P$ vanishes on $\mathcal R$.
Hence
\[
        \{0\}\neq P\mathcal R\subseteq\mathcal K_P\cap\mathcal R,
\]
a contradiction. Therefore
\[
        \mathcal H=\bigoplus_{P\in J(S)}\mathcal K_P.
\]

For each $P$ with $\mathcal K_P\neq\{0\}$, let $\pi_P$
be the restriction of $\pi$ to $\mathcal K_P$. An idempotent $Q\in J(S)$ acts on $\mathcal K_P$ as
the identity when $P\leq Q$, and as zero otherwise, since then $PQ<P$.
Restricting the operators $\pi(g)Q$ to $\mathcal K_P$
therefore gives
\[
        S_{\pi_P}\subseteq\pi_P(G)\cup\{0\}.
\]
Put $G_P=G/\ker\pi_P$. This is a Polish sealed group and hence minimal.
The induced faithful representation identifies $G_P$ topologically
with $\pi_P(G)$. Here the weak and strong operator topologies agree on the
unitary group.

If $0\notin S_{\pi_P}$, then $S_{\pi_P}=\pi_P(G)$ is compact, and
$G_P$ is compact. Peter--Weyl theory decomposes this summand into
finite-dimensional representations. Otherwise
$S_{\pi_P}=\pi_P(G)\cup\{0\}$ is the one-point compactification of
$G_P$. Its coefficient functions vanish at the added point, so
$\pi_P$ is mixing on $G_P$. This proves the assertion.
\end{proof}

\subsection{The converse implications}

We first prove that the Eberlein property implies sealedness.
The argument passes to a quotient with noncompact centre.

\begin{lemma}\label{lem:eberlein-quotient}
Every Hausdorff quotient of a locally compact Eberlein group is Eberlein.
\end{lemma}

\begin{proof}
Let $Q=H/N$, where $N\lhd H$ is closed, and let $f\in\operatorname{WAP}(Q)$.
Its pullback $F$ belongs to $\operatorname{WAP}(H)=E(H)$. Approximate $F$
uniformly by a matrix coefficient
$b(g)=\langle\pi(g)\xi,\eta\rangle$.
Let $P_N$ be the orthogonal projection onto $\mathcal H^N$.
The vector $P_N\xi$ lies in the norm-closed convex hull of $\pi(N)\xi$.
Since $F$ is right $N$-invariant, replacing $\xi$ by $P_N\xi$ does not
increase the uniform approximation error: apply the same convex
combinations to the right translates of $b$ and pass to the norm limit.
Normality of $N$ makes $\mathcal H^N$ $H$-invariant, and
\[
        \langle\pi(g)P_N\xi,\eta\rangle
        =\langle\pi(g)P_N\xi,P_N\eta\rangle
\]
is the pullback of a coefficient of a unitary representation of $Q$.
Thus $f\in E(Q)$.
\end{proof}

\begin{prop}\label{prop:eberlein-implies-sealed}
If $G$ is Eberlein, then $G$ is sealed.
\end{prop}

\begin{proof}
Suppose otherwise. By \cite[Theorem~8.1]{bader2023homomorphic}, there is
a normal unipotent $k$-subgroup $\mathbf N\lhd\mathbf G$ for which
\[
        Q=(\mathbf G/\mathbf N)(k)\cong G/\mathbf N(k)
\]
has noncompact centre $Z$. The group $Q$ is Eberlein by
Lemma~\ref{lem:eberlein-quotient}.
Every weakly almost periodic function on the closed central subgroup
$Z$ extends to one on $Q$: this is the central-subgroup case of
\cite[Theorem~1]{cowling1979restrictions}, whose conjugation-continuity
condition is automatic for a central subgroup. Restriction also maps
$E(Q)$ into $E(Z)$. Hence $Z$ would be Eberlein, contradicting the
Rudin--Ramirez theorem that a noncompact locally compact abelian group
is not Eberlein \cite{MR0102705,ramirez1968weakly}.
\end{proof}

For the representation-theoretic converse, the obstruction is a
nonzero scalar weak limit whose modulus is strictly less than one.

\begin{lemma}\label{lem:scalar-weak-limit-obstruction}
Suppose that a nonzero unitary representation $\sigma$ of a locally
compact group $H$ admits a sequence $(h_n)$ such that
\[
        \sigma(h_n)\longrightarrow cI
        \quad\text{in the weak operator topology},
        \qquad 0<|c|<1.
\]
Then $\sigma$ cannot be a direct sum of finite-dimensional
representations and representations mixing modulo their respective
kernels.
\end{lemma}

\begin{proof}
Suppose that such a decomposition exists, and let $\tau$
be a nonzero summand. Then $\tau(h_n)\to cI$ weakly.
If $\tau$ is finite-dimensional, weak convergence is
norm convergence, so $cI$ is unitary, contradicting
$0<|c|<1$.

Otherwise, $\tau$ is mixing on $H/N$, where
$N=\ker\tau$. If $(h_nN)$ has a convergent subnet,
continuity implies that the corresponding operators
converge weakly to a unitary operator, again contradicting
$\tau(h_n)\to cI$. If there is no convergent subnet,
$(h_nN)$ escapes every compact subset of $H/N$.
Mixing then gives $\tau(h_n)\to0$ weakly, contradicting
$c\neq0$.
\end{proof}
\begin{lemma}\label{lem:central-scalar-limit}
Let $H=\mathbf H(k)$ for a connected affine $k$-algebraic group.
If $Z(H)$ is noncompact, there is a nonzero unitary representation
$\sigma$ of $H$ and a sequence $(z_n)$ in $Z(H)$ such that
\[
        \sigma(z_n)\longrightarrow cI
        \quad\text{weakly},\qquad 0<|c|<1.
\]
\end{lemma}
\begin{proof}
First, $Z(H)$ contains a closed subgroup $A$ topologically
isomorphic to $\mathbb R$, $\mathbb Z$, or $\mathbb Q_p$
for some prime $p$. Indeed, $Z(H)=Z(\mathbf H)(k)$, and
$Z(\mathbf H)^\circ(k)$ has finite index in $Z(H)$.
It is therefore noncompact: otherwise $Z(H)$ would be
a finite union of compact cosets. A connected commutative
affine group in characteristic zero is a product of a
torus and a vector group. A nonzero vector factor contains
a closed copy of $\mathbb R$ or $\mathbb Q_p$. If the
torus factor is noncompact, it contains a nontrivial
$k$-split torus. In the Archimedean case, the positive
real subgroup of $k^\times$ is a closed copy of
$(\mathbb R,+)$ via $t\mapsto e^t$. In the non-Archimedean
case, the subgroup generated by a uniformiser is a closed
copy of $\mathbb Z$.

We construct a unitary representation $\tau$ of $A$
and a sequence $(a_n)$ such that $\tau(a_n)\to cI$
weakly, with $0<|c|<1$. Let
$\Omega=\{0,1\}^{\mathbb N}$, and write
$\omega=(\omega_1,\omega_2,\ldots)$ for its elements.
Equip $\Omega$ with the product probability measure $\mu$
under which the coordinate functions
$\omega\mapsto\omega_r$ are independent and satisfy
\[
        \mu(\omega_r=0)=\frac34,
        \qquad
        \mu(\omega_r=1)=\frac14.
\]
Write $\mathbb E$ for expectation with respect to $\mu$,
and put $\mathcal K=L^2(\Omega,\mu)$.

For an integer $q\geq2$, we have
\[
        \mathbb E\left(e^{2\pi i\omega_r/q^r}\right)
        =\frac34+\frac14e^{2\pi i/q^r}.
\]
Independence and dominated convergence therefore give
\[
        \mathbb E\left(
        \exp\left(2\pi i\sum_{r=1}^{\infty}\omega_rq^{-r}\right)
        \right)
        =
        \prod_{r=1}^{\infty}
        \left(\frac34+\frac14e^{2\pi i/q^r}\right)
        \coloneqq c_q.
\]
We check that $0<|c_q|<1$. Every factor has modulus at
least $1/2$, and
\[
        \left|
        \left(\frac34+\frac14e^{2\pi i/q^r}\right)-1
        \right|
        \leq\frac{\pi}{2q^r}.
\]
Thus the logarithms of the factors sufficiently close
to $1$ form an absolutely convergent series, so the
product has a nonzero limit. All factors have modulus
at most one, and the first has modulus strictly less
than one. Hence $0<|c_q|<1$.

For $A=\mathbb R$ or $\mathbb Z$, define
\[
        X(\omega)=\sum_{j=1}^{\infty}\omega_j3^{-j},
        \qquad
        (\tau(a)F)(\omega)=e^{2\pi i aX(\omega)}F(\omega).
\]
Take $a_n=3^n$. The first $n$ terms of $3^nX$ are
integers, so
\[
        e^{2\pi i3^nX(\omega)}
        =\exp\left(2\pi i
          \sum_{r=1}^{\infty}\omega_{n+r}3^{-r}\right).
\]
This function depends only on digits after the $n$th
digit, and its expectation is $c_3$. Thus, if $F_1,F_2$
depend only on the first $m$ digits, independence gives
\[
\begin{aligned}
        \langle\tau(a_n)F_1,F_2\rangle
        &=\mathbb E\left(e^{2\pi i3^nX}
                         F_1\overline{F_2}\right)\\
        &=c_3\mathbb E\left(F_1\overline{F_2}\right)
         =c_3\langle F_1,F_2\rangle
        \qquad(n\geq m).
\end{aligned}
\]
Functions depending on finitely many digits are dense
in $\mathcal K$. Since $\|\tau(a_n)\|=1$, approximation
gives $\tau(a_n)\to c_3I$ weakly.

For $A=\mathbb Q_p$, let $\psi$ be the continuous additive
character with kernel $\mathbb Z_p$ satisfying
$\psi(p^{-r})=e^{2\pi i/p^r}$, and define
\[
        X_p(\omega)=\sum_{j=1}^{\infty}\omega_jp^{j-1},
        \qquad
        (\tau(a)F)(\omega)=\psi(aX_p(\omega))F(\omega).
\]
Take $a_n=p^{-n}$. The terms with $j>n$ belong to
$\mathbb Z_p$ after multiplication by $p^{-n}$, so
\[
        \psi(p^{-n}X_p(\omega))
        =\exp\left(2\pi i
          \sum_{j=1}^{n}\omega_jp^{j-1-n}\right).
\]
Suppose that $F_1,F_2$ depend only on the first $m$
digits. For $n\geq m$, put
\[
        b_{n,m}(\omega)
        =\exp\left(2\pi i
          \sum_{j=1}^{m}\omega_jp^{j-1-n}\right).
\]
For fixed $m$, these functions converge uniformly to $1$.
Averaging over the independent digits $m<j\leq n$ gives
\[
\begin{aligned}
        \langle\tau(p^{-n})F_1,F_2\rangle
        &=
        \prod_{r=1}^{n-m}
        \left(\frac34+\frac14e^{2\pi i/p^r}\right)
        \mathbb E\left(b_{n,m}F_1\overline{F_2}\right)\\
        &\longrightarrow c_p\langle F_1,F_2\rangle.
\end{aligned}
\]
Again, density and $\|\tau(p^{-n})\|=1$ give
$\tau(p^{-n})\to c_pI$ weakly.

The representations just constructed are strongly
continuous: for $A=\mathbb Z$ this is automatic, and
for $A=\mathbb R$ or $\mathbb Q_p$ it follows by
dominated convergence.

Finally, put $\sigma=\operatorname{Ind}_A^H\tau$.
Since $A$ is central, $\Delta_H|_A=\Delta_A=1$, so
$H/A$ carries an $H$-invariant measure. Using a measurable
section, realise $\sigma$ on $L^2(H/A,\mathcal K)$.
Centrality of $A$ gives
\[
        (\sigma(a)F)(x)=\tau(a)F(x)
        \qquad(a\in A,\ x\in H/A).
\]
Thus $\sigma|_A$ is a Hilbert multiple of $\tau$.
The convergence $\tau(a_n)\to cI$ implies
$\sigma(a_n)\to cI$ weakly, first on elementary tensor
vectors and then, by density and uniform boundedness,
on the whole space. Here $c=c_3$ or $c_p$, so
$0<|c|<1$. Taking $z_n=a_n\in A\subseteq Z(H)$ proves the assertion.
\end{proof}

\begin{prop}\label{prop:representation-decomposition-implies-sealed}
If every unitary representation of $G$ has the decomposition in
Theorem~\ref{thm:harmonic-equivalences}(3), then $G$ is sealed.
\end{prop}

\begin{proof}
The asserted decomposition property passes to Hausdorff quotients:
lift a representation to $G$, decompose it, and observe that the
kernel of the quotient map acts trivially on every summand. The
quotients by the kernels of corresponding summands are topologically
the same.

If $G$ were not sealed, \cite[Theorem~8.1]{bader2023homomorphic}
would provide a normal unipotent $k$-subgroup $\mathbf N$ for which
$H=(\mathbf G/\mathbf N)(k)\cong G/\mathbf N(k)$ has noncompact
centre. Lemma~\ref{lem:central-scalar-limit} would then give a
representation of $H$ that violates the decomposition property, by
Lemma~\ref{lem:scalar-weak-limit-obstruction}. This is a contradiction.
\end{proof}

\begin{proofof}{Theorem~\ref{thm:harmonic-equivalences}}
If $G$ is sealed, every Hausdorff quotient is Polish and sealed, hence
minimal, and has compact centre
\cite[Lemmas~4.12 and~4.14]{bader2023homomorphic}. This proves
$(2)\Rightarrow(1)$. Conversely, (1) applies in particular to all
quotients by normal unipotent $k$-subgroups. These are the full groups
of $k$-points of the corresponding algebraic quotients, so
\cite[Theorem~8.1]{bader2023homomorphic} gives $(1)\Rightarrow(2)$.

Propositions~\ref{prop:sealed-implies-eberlein} and
\ref{prop:eberlein-implies-sealed} prove $(2)\Longleftrightarrow(4)$.
Propositions~\ref{prop:sealed-representation-decomposition} and
\ref{prop:representation-decomposition-implies-sealed} prove
$(2)\Longleftrightarrow(3)$.
\end{proofof}

\begin{proofof}{Theorem~\ref{mthm:main}}
Apply the equivalence of conditions (2), (3), and (4) in
Theorem~\ref{thm:harmonic-equivalences}.
\end{proofof}

\bibliographystyle{amsplain}
\bibliography{references}
\end{document}